\documentclass[12pt,makeidx]{amsart}
\usepackage{pdfpages}
\usepackage{amsmath}
\usepackage{amsfonts}
\usepackage{amssymb}
\usepackage{xypic}
\usepackage{graphicx}
\usepackage{url}

\newtheorem{thm}{Theorem}[section]

\newtheorem{cor}[thm]{Corollary}

\newtheorem{remark}[thm]{Remark}
\newtheorem{lemma}[thm]{Lemma}
\newtheorem{prop}[thm]{Proposition}
\newtheorem{prob}[thm]{Problem}
\newtheorem{exam}[thm]{Example}
\newtheorem{defn}[thm]{Definition}

\DeclareMathOperator{\cart}{\,\square\hspace{0.07em}}
\DeclareMathOperator{\CC}{\mathbb{C}}
\DeclareMathOperator{\Hom}{Hom}
\DeclareMathOperator{\lc}{\stackrel{lc}{\to}}

\DeclareMathOperator{\ZZ}{\mathbb{Z}}

\newcommand{\<}{\langle}
\renewcommand{\>}{\rangle}
\newcommand{\bb}[1]{\mathbb{#1}}
\newcommand{\cl}[1]{\mathcal{#1}}

\usepackage{tikz}
\usepackage{xcolor}

\numberwithin{equation}{section}

\def\ben{\begin{enumerate} }
\def\een{\end{enumerate} }

 \def\beq{\begin{equation}}
 \def\eeq{\end{equation}}

\def\ssec{\subsection}

\def\bs{\bigskip}

\makeindex
\newcommand{\df}[1]{{\bf{#1}}{\index{#1}}}

\begin{document}

\title[]{Algebras, Synchronous Games and Chromatic Numbers of Graphs}
\author[J.~W.~Helton]{J. William Helton}
\address{Department of Mathematics, UCSD}
\email{helton@ucsd.edu}
\author[K. P. Meyer]{Kyle P. Meyer}
\address{Department of Mathematics, UCSD}
\email{kpmeyer@ucsd.edu}
\thanks{The first and second-named authors have been supported in part by NSF DMS-1500835  }
\author[V.~I.~Paulsen]{Vern I.~Paulsen}
\address{Department of Pure Mathematics, University of Waterloo,
Waterloo, ON, Canada}
\email{vpaulsen@uwaterloo.ca}
\author[M.~Satriano]{Matthew Satriano}
\thanks{The third and fourth-named authors are supported in part by NSERC Discovery Grants}
\email{msatriano@uwaterloo.ca}

\begin{abstract}  We associate to each synchronous game an algebra whose representations determine if the game has a perfect deterministic strategy, perfect quantum strategy or one of several other strategies. When applied to the graph coloring game, this leads to characterizations in terms of properties of an algebra of various quantum chromatic numbers that have been studied in the literature. This allows us to develop a correspondence between various chromatic numbers of a graph and ideals in this algebra which can then be approached via Gr\"obner basis methods.

\end{abstract}

\maketitle


\section{Introduction}

Given a graph $G$ and a natural number $c$, there is a game called the {\it c-coloring game} of $G$ and it is known that there is a perfect deterministic strategy for this game if and only if $G$ has a c-coloring. Thus, the {\it chromatic number} of $G$, $\chi(G)$ can be characterized as the least integer for which a perfect deterministic strategy exists for the c-coloring game.  This led researchers to consider various kinds of probabilistic strategies for games, especially cases where the probabilities arose from the random outcomes of quantum experiments that were in finite dimensional entangled states, called {\it quantum strategies}. The least $c$ for which the c-coloring game has such a perfect strategy became known as the {\it quantum chromatic} number of a graph, denoted $\chi_q(G)$. For an introduction to this literature, see \cite{CMRSSW}, \cite{MR}. There are some open questions about the proper model for the set of probabilities that arise from entangled quantum experiments and this has led to the study of several, potentially different, definitions of quantum chromatic numbers denoted $\chi_{qa}(G)$ and $\chi_{qc}(G)$, see \cite{PT} and \cite{PSSTW}.

However, while there are many graphs $G$ for which $\chi_q(G)\neq\chi(G)$, it is not yet known whether or not $\chi_q$, $\chi_{qa}$, and $\chi_{qc}$ can assume different values.
Furthermore, there are currently no general algorithms for computing these three quantum chromatic numbers.

There is also an alternative characterization of $\chi(G)$ in terms of graph homomorphisms. If $K_c$ denotes the complete graph on $c$ vertices, then $\chi(G)$ is the least integer $c$ for which a graph homomorphism exists from $G$ to $K_c$. Again there is a corresponding {\it graph homomorphism game} and several, potentially different definitions of {\it quantum graph homomorphisms} defined in terms of the existence of perfect quantum strategies for the graph homomorphism game. In each case, the least c for which there exists a perfect quantum strategy for the graph homomorphism game from $G$ to $K_c$ is the corresponding quantum chromatic number. For more on this literature see, \cite{MR}, \cite{MR2}, and \cite{Ro}.

Given graphs $G$ and $H$, C. Ortiz and the third author \cite{OP2} affiliated a *-algebra to this pair, denoted $\cl A(G,H)$, and characterized the existence of graph homomorphisms from $G$ to $H$ in terms of representations of this algebra: \cite{OP2} proves that $\cl A(G,H)$ has a non-zero homomorphism into the complex numbers if and only if there exists a classical graph homomorphism from $G$ to $H$. Because the problem of determining if $\chi(G) \le 3$ is an NP-complete problem, the results of \cite{OP2} allow us to deduce that the problem of determining if $\cl A(G, K_3)$ has a non-trivial homomorphism into the complex numbers is an NP-complete problem.

Similarly, \cite{OP2} also shows that there exists a {\it quantum graph homomorphism},
as defined by \cite{MR} and \cite{Ro}, from $G$ to $H$ if and only if $\cl A(G,H)$
has a non-zero *-homomorphism into the matrices, that is, a non-zero finite dimensional representation.  A result of \cite{Ji} shows that the problem of determining if $\chi_q(G) \le 3$ is an NP-hard problem. Hence, the problem of determining if $\cl A(G, K_3)$ has any non-zero finite dimensional representations is NP-hard.

Currently, the computational complexity of these other variations of the quantum chromatic number is unknown. The results of \cite{OP2} characterize the existence of these other types of quantum graph homomorphisms in terms of the existence of various types of traces on the algebra $\cl A(G,H)$.
Thus, questions about the existence or non-existence of these various types of
graph homomorphisms, and consequently quantum colorings, is reduced to questions about these algebras.

\subsection{New types of chromatic numbers}
This leads us to a more detailed study of these algebras, which is the main topic of this paper. Since the existence of these various types of representations can be characterized in terms of whether or not certain ideals in these algebras are proper, we are naturally led to study new chromatic numbers of a graph determined by the least integer $c$ so that the type of ideal in which we are interested is proper. Thus, whenever a certain type of ideal is proper, we can use that as the definition of a new type of graph homomorphism and obtain a new {\it algebraic} chromatic number of a graph, which we then try to relate to prior chromatic numbers. Three of these new parameters are denoted $\chi_{alg}(G)$, $\chi_{hered}(G)$, and $\chi_{lc}(G)$. One goal of this paper is to study the properties of these new chromatic numbers.

In general, no algorithms are known for computing the quantum chromatic numbers of a graph. One advantage to $\chi_{alg}(G)$ is that its computation reduces to a  Gr\"obner basis problem in a non-commutative algebra. However, using a machine-aided proof, we are able to show that $\chi_{alg}(G) \le 4$ for every graph, which is not the case for any of the earlier chromatic numbers.

On the other hand, we know of no graphs that separate $\chi_{hered}$ from the other quantum chromatic numbers, while computing this parameter can be done approximately with a combination of algebra and optimization (semidefinite programming).

Since all the earlier notions of quantum graph homomorphism and the corresponding quantum chromatic numbers, were defined in terms of the existence of {\it perfect quantum strategies} for certain games, we begin by identifying a certain family of games for which we can carry out the construction of \cite{OP2} and affiliate an algebra to the game such that the existence of various types of representations of the algebra determines whether or not various types of perfect quantum strategies exist for the game. For this family of games it is possible that all of the notions of perfect quantum strategies coincide and are equivalent to a certain {\it hereditary} ideal in this algebra being proper.

\subsection{Outline of paper}

In Section \ref{sec:def}, we provide some background on games and strategies, introduce the family of {\it synchronous games} which are the games for which we can extend the results of \cite{OP2} and construct an algebra.

In Section \ref{sec:alg-sync-game}, we introduce the *-algebra of a synchronous game, the chromatic numbers $\chi_{alg}(G)$ and $\chi_{hered}(G)$ and prove some of their properties.

In Section \ref{sec:123colors}, we focus on the case of 1, 2, and 3 colors.

In Section \ref{sec:alg}, we show that removing the assumption that the algebra be a free *-algebra instead of just a free algebra, changes nothing essential. We also show that it is enough to study these algebras over $\bb Q$ instead of $\bb C$. This allows us to use a Gr\"obner basis approach.

In Section \ref{sec:4colors}, we present the details of the machine-assisted proof that $\chi_{alg}(G) \le 4$ for all graphs.

In Sections \ref{sec:loc} and \ref{sec:loc properties}, we introduce and study the {\it locally commuting} chromatic number $\chi_{lc}(G)$.

\section{Synchronous Games and Strategies}
\label{sec:def}

We lay out some definitions and a few basic properties of games and strategies. We will primarily be concerned with the $c$-coloring game and the graph homomorphism game.

\ssec{Definitions of games and strategies}

By a \df{two-person  finite input-output game} we mean a tuple
 $\cl G=(I_A, I_B, O_A, O_B, \lambda)$ where $I_A, I_B, O_A, O_B$ are finite sets and
\[ \lambda: I_A \times I_B \times O_A \times O_B \to \{ 0,1 \} \]
is a function that represents the rules of the game, sometimes called the predicate.
The sets $I_A$ and $I_B$ represent the inputs that Alice and Bob can receive,
and the sets $O_A$ and $O_B$, represent the outputs that Alice and Bob can produce,
respectively.  A referee selects a pair $(v,w) \in I_A \times I_B$, gives Alice $v$ and Bob $w$, and they then produce outputs (answers), $a \in O_A$ and $b \in O_B$, respectively.
They win the game if $\lambda(v,w,a,b) =1$ and lose otherwise.
Alice and Bob are allowed to know the sets and the function $\lambda$
and cooperate before the game to produce a strategy for providing outputs,
but while producing outputs, Alice and Bob only know their own inputs
and are not allowed to know the other person's input. Each time that they are given an input and produce an output is referred to as a \df{round} of the game.

We call such a game \df{synchronous} provided that: (i) Alice and Bob have the same input sets and the same output sets, which we denote by $I$ and $O$, respectively, and (ii) $\lambda$ satisfies:
\[ \forall v \in I, \,\, \lambda(v,v,a,b) = \begin{cases} 0 & a \ne b\\ 1 & a=b \end{cases},\]
that is, whenever Alice and Bob receive the same inputs then they must produce the same outputs. To simplify notation we write a synchronous game as $\cl G= (I,O, \lambda)$.

A {\it graph} $G$ is specified by a vertex set $V(G)$ and an edge set $E(G) \subseteq V(G) \times V(G)$, satisfying $(v,v) \notin E(G)$ and $(v,w) \in E(G) \implies (w,v) \in E(G)$. The \df{c-coloring game} for $G$ has inputs $I_A=I_B = V(G)$ and outputs $O_A=O_B = \{ 1,...,c \}$ where the outputs are thought of as different colors. They win provided that whenever Alice and Bob receive adjacent vertices, i.e., $(v,w) \in E$, their outputs are different colors and when they receive the same vertex they must output the same color. Thus,
$(v,w) \in E(G) \implies \lambda(v,w,a,a)=0, \, \forall a$, $\lambda(v,v,a,b)=0, \, \forall v \in V(G), \, \forall a \ne b$ and the rule function is equal to 1 for all other tuples.  It is easy to see that this is a synchronous game.

Given two graphs $G$ and $H$, a {\it graph homomorphism from G to H} is a function $f:V(G) \to V(H)$ with the property that $(v,w) \in E(G) \implies (f(v), f(w)) \in E(H)$.
The \df{graph homomorphism game} from $G$ to $H$ has inputs $I_A=I_B = V(G)$ and outputs $O_A=O_B = V(H).$  They win provided that whenever Alice and Bob receive inputs that are an edge in $G$, then their outputs are an edge in $H$ and that whenever Alice and Bob receive the same vertex in $G$ they produce the same vertex in $H$. This is also a synchronous game.

A \df{deterministic strategy} for a game is a pair of functions, $h: I_A \to O_A$ and $k:I_B \to O_B$ such that if Alice and Bob receive inputs $(v,w)$ then they produce outputs $(h(v), k(w))$. A deterministic strategy  wins every round of the game if and only if
\[ \forall  (v,w) \in I_A \times I_B,   \lambda(v,w, h(v), k(w)) =1.\] Such a strategy is called a \df{perfect deterministic strategy}.

It is not hard to see that for a synchronous game, any perfect deterministic strategy must satisfy, $h=k$.  In particular, a perfect deterministic strategy for the c-coloring game for $G$ is a function $h: V(G) \to \{ 1,...,c \}$ such that $(v,w) \in E(G) \implies h(v) \ne h(w)$.  Thus, a perfect deterministic strategy is precisely a c-coloring of  $G$.
Similarly, a perfect deterministic strategy for the graph homomorphism is precisely a graph homomorphism.

Finally, it is not difficult to see that if $K_c$ denotes the complete graph on $c$ vertices then a graph homomorphism exists from $G$ to $K_c$ if and only if $G$ has a c-coloring. This is because any time $(v,w) \in E(G)$ then a graph homomorphism must send them to distinct vertices in $K_c$. Indeed, the rule function for the c-coloring game is exactly the same as the rule function for the graph homomorphism game from $G$ to $K_c$.

A \df{random strategy} for such a game is a conditional probability density $p(a,b|v,w)$, which represents the probability that, given inputs $(v,w) \in I_A \times I_B$, Alice and Bob produce outputs $(a,b) \in O_A \times O_B$.  Thus, $p(a,b|v,w) \ge 0$ and for each $(v,w),$
\[ \sum_{a \in O_A, b \in O_B} p(a,b|v,w) =1.\]

In this paper we identify random strategies with their conditional probability density, so that a random strategy will simply be a conditional probability density $p(a,b|v,w)$.

A random strategy is called \df{perfect} if
\[ \lambda(v,w,a,b)=0 \implies p(a,b|v,w) =0, \, \forall (v,w,a,b) \in I_A \times I_B \times O_A \times O_B.\]
Thus, for each round, a perfect strategy gives a winning output with probability 1.

We next discuss \df{local} random strategies, which are also sometimes called \df{classical}, meaning not quantum. They are obtained as follows: Alice and Bob share a probability space $(\Omega, P)$, for each input $v \in I_A$, Alice has a random variable,
$f_v: \Omega \to O_A$ and for each input $w \in I_B$, Bob has a random variable, $g_w: \Omega \to O_B$ such that for each round of the game, Alice and Bob will evaluate their random variables at a point $\omega \in \Omega$ via a formula that has been agreed upon in advance. This yields conditional probabilities,
\[ p(a,b|v,w) = P(\{ \omega \in \Omega \mid f_v(\omega) =a, g_w(\omega) =b \}).\]
This will be a perfect strategy if and only if
\[ \forall (v,w), \, P(\{ \omega \in \Omega \, \mid \lambda(v,w, f_v(\omega), g_w(\omega)) =0 \}) =0,\]
or equivalently,
\[ \forall (v,w), \, P(\{ \omega \in \Omega \, \mid \lambda(v,w, f_v(\omega), g_w(\omega)) =1 \}) =1.\]

If we have a perfect local strategy and set
\[ \Omega_1 = \cap_{v \in I_A, w \in I_B} \{ \omega \in \Omega \, \mid \lambda(v,w, f_v(t), g_w(t)) =1 \},\]
then $P(\Omega_1) =1$ since $I_A$ and $I_B$ are finite sets; in particular, $\Omega_1$ is non-empty. If we choose any $\omega \in \Omega_1$ and set $h(v) = f_v(\omega)$ and $k(w) = g_w(\omega)$, then it is easily checked that this is a perfect deterministic strategy.

Thus, a perfect classical random strategy exists if and only if a perfect deterministic strategy exists.  An advantage to using a perfect classical random strategy over a perfect deterministic strategy, is that it is difficult for an observer to find a deterministic strategy even after observing the outputs of many rounds.

The idea behind \df{nonlocal games} is to allow a larger set of conditional probabilities, namely, those that can be obtained by allowing Alice and Bob to run quantum experiments to obtain their outputs.

Definitions of these various sets of probability densities, including loc, q, qa, qc, vect, nsb can be found in \cite[Section~6]{OP2} or \cite{PSSTW}, so we will avoid repeating them here.
We only remark that if for $t\in \{ loc, q, qa, qc, vect, nsb \}$ we use $C_t$ to denote the corresponding set of conditional probabilities, then it is known that
\[ C_{loc} \subsetneq C_q \subseteq C_{qa} \subseteq C_{qc} \subsetneq C_{vect} \subsetneq C_{nsb}.\]

The sets $C_q$, $C_{qa}$ and $C_{qc}$ represent three potentially different mathematical models for the set of all probabilities that can arise as outcomes from entangled quantum experiments. The question of whether or not $C_{qa}=C_{qc}$ for any number of experiments and any number of outputs is known to be equivalent to Connes' embedding conjecture due to results of \cite{Oz13}.

We say that $p(a,b|v,w)$ is a \df{perfect t-strategy} for a game provided that it is a perfect strategy that belongs to the corresponding set of probability densities.

Given a graph $G$ we set $\chi_t(G)$ equal to the least $c$ for which there exists a perfect t-strategy for the c-coloring game for $G$. The above inclusions imply that
\[ \chi(G) = \chi_{loc}(G) \ge \chi_q(G) \ge \chi_{qa}(G) \ge \chi_{qc}(G) \ge \chi_{vect}(G) \ge \chi_{nsb}(G).\]
Currently, it is unknown if there are any graphs that separate $\chi_q(G), \chi_{qa}(G)$ and $\chi_{qc}(G)$ or whether these three parameters are always equal.
Examples of graphs are known for which $\chi(G) > \chi_q(G)$, for which $\chi_{qc}(G) > \chi_{vect}(G)$ and for which $\chi_{vect}(G) > \chi_{nsb}(G)$.
For details, see \cite{CMRSSW}, \cite{PT} and \cite{PSSTW}.
Other versions of quantum chromatic type  graph parameters appear
in \cite{BLP} and a comparison of those parameters with 
$\chi_{qa}$ and $ \chi_{qc}$ can be found in \cite[Section 1]{BLP}.

Similarly, we say that there is a {\it t-homomorphism} from $G$ to $H$ if and only if there exists a perfect t-strategy for the graph homomorphism game from $G$ to $H$ and it is unknown if q-homomorphisms, qa-homomorphisms and qc-homomorphisms are distinct or coincide.

Finally, we close this section by showing that it is enough to consider so-called symmetric games. Note that in a synchronous game there is no requirement that  $\lambda(v,w,a,b) =0 \implies \lambda(w,v,b,a)=0$. That is, the rule function does not need to be \df{symmetric} in this sense.  The following shows that it is enough to consider synchronous games with this additional symmetry.

Given $\cl G = (I,O, \lambda)$ a synchronous game,  we define
$\lambda_s: I \times I \times O \times O \to \{ 0,1 \}$ by setting
$\lambda_s(v,w,a,b) = \lambda(v,w,a,b) \lambda(w,v,b,a)$ and set $\cl G_s= (I,O, \lambda_s)$. Then it is easily seen that $\cl G_s$ is a synchronous game with the property that $\lambda_s(v,w,a,b)=0 \iff \lambda_s(w,v,b,a)=0$.

\ssec{A few properties of strategies}
In the remainder of this section, we prove the following slight extension of \cite{PT}.

\begin{prop} Let $\cl G= (I,O, \lambda)$ be a synchronous game and let
$p(a,b|v,w)= \langle h_{v,a}, k_{w,b} \rangle$ be a perfect vect-strategy for $\cl G$,
 where the vectors $h_{v,a}$ and $k_{w,b}$ are as in the definition of a vector correlation(see \cite[6.15]{OP2}). Then $h_{v,a} = k_{v,a}, \forall v \in I,  a \in O$.
\end{prop}
\begin{proof} By definition, for each $v \in I$ the vectors $\{h_{v,a}: a \in O \}$ are mutually orthogonal and $\{ k_{v,a}: a \in O \}$ are mutually orthogonal. So,
\begin{multline*}
 1 = \sum_{a,b \in O} p(a,b|v,v) = \sum_{a \in O} p(a,a|v,v) = \sum_{a \in O} \langle h_{v,a}, k_{v,a} \rangle \le \\ \sum_{a \in O}  \|h_{v,a}\| \|k_{v,a}\| \le \big( \sum_{a \in O} \|h_{v,a}\|^2 \big)^{1/2} \big( \sum_{a \in O} \|k_{v,a}\|^2 \big)^{1/2}=1.\end{multline*}
Thus, the inequalities are equalities, which forces $h_{v,a} = k_{v,a}$ for all $v \in I$ and all $a \in O$.
\end{proof}
\begin{cor} Let $\cl G=(I,O,\lambda)$ be a synchronous game and let $t \in \{ loc, q, qa, qc, vect \}$. If $p(a,b|v,w)$ is a perfect t-strategy for $\cl G$, then $p(a,b|v,w) = p(b,a|w,v)$ for all $v,w \in I$ and all $a,b \in O$.
\end{cor}
\begin{proof} If $p(a,b|v,w)$ is a perfect t-strategy, then it is a perfect vect-strategy and hence there exist vectors as in the definition such that,
\[ p(a,b|v,w) = \langle h_{v,a}, h_{w,b} \rangle = \langle h_{w,b}, h_{v,a} \rangle = p(b,a|w,v),\]
where the middle equality follows since the inner products are assumed to be non-negative.
\end{proof}

This corollary readily yields the following result.

\begin{prop}\label{prop:reduce-to-synch-games}
 Let $\cl G=(I,O, \lambda)$ be a synchronous game and let $t \in \{ loc, q,qa,qc, vect \}$. Then $p(a,b|v,w)$ is a perfect t-strategy for $\cl G$ if and only if $p(a,b|v,w)$ is a perfect t-strategy for $\cl G_s$.
\end{prop}
\section{The *-algebra of a synchronous game}
\label{sec:alg-sync-game}
We begin by constructing a *-algebra, defined by generators and relations, that is affiliated with a synchronous game. The existence or non-existence of various types of perfect quantum strategies for the game then corresponds to the existence or non-existence of various types of representations of this algebra. This leads us to examine various ideals in the algebra.

\ssec{Relations generators and the basic *-algebra}
Let $\cl G= ( I,O, \lambda)$ be a synchronous game and assume that the cardinality of $I$ is $|I|=n$ while the cardinality of $O$ is $|O|=m$. We will often identify $I$ with $\{ 0,..., n-1 \}$ and $O$ with $\{ 0,..., m-1 \}$. We let $\bb F(n,m)$ denote the free product of $n$ copies of the cyclic group of order $m$ and let $\bb C[\bb F(n,m)]$ denote the complex *-algebra of the group.
We regard the group algebra as both a *-algebra, where for each group element $g$
we have $g^* = g^{-1}$,
and 
as an (incomplete) inner product space,
with the group elements forming an orthonormal set and the inner product is given by
\[ \langle f, h \rangle = \tau(fh^*),\]
where $\tau$ is the trace functional.

For each $v \in I$ we have a unitary generator  $u_v \in \bb C[\bb F(n,m)]$ such that $u_v^m = 1$. If we set $\omega = e^{2 \pi i/m}$ then the eigenvalues of each $u_v$ is the set $\{ \omega^a: 0 \le a \le m-1 \}$.  The orthogonal projection onto the eigenspace corresponding to $\omega^a$ is given by
\beq
\label{eq:efromu}
e_{v,a} = \frac{1}{m} \sum_{k=0}^{m-1} \big( \omega^{-a} u_v \big)^k,
\eeq
and these satisfy
\[ 1 = \sum_{a=0}^{m-1} e_{v,a} \text{ and } u_v = \sum_{a=0}^{m-1} \omega^a e_{v,a}.\]
The set $\{ e_{v,a}: v \in I, 0 \le a \le m-1 \}$ is another set of generators for $\bb C[\bb F(n,m)]$.

We let $\cl I(\cl G)$  \index{$\cl I(\cl G)$}
denote the 2-sided *-ideal in $\bb C[\bb F(n,m)]$ generated by the set
\[ \{ e_{v,a}e_{w,b} \, | \ \lambda(v,w,a,b)=0 \} \]
and refer to it as \df{the ideal of the game $\cl G$}. We define the \df{*-algebra of $\cl G$} to be the quotient
\[ \cl A(\cl G) = \bb C[\bb F(n,m)]/\cl I(\cl G) .\]


A familiar case occurs when we are given two graphs $G$ and $H$ and $\cl G$ is the graph homomorphism game from $G$ to $H$.  Then $\cl A(\cl G) =\cl A(G,H)$, where the algebra on the right hand side is the algebra introduced in \cite{OP2}, so we shall continue that notation in this instance. Recall that $\cl A(G, K_c)$ is then the algebra of the c-coloring game for $G$.

\begin{defn} We say that a game has a perfect algebraic strategy if $\cl A(\cl G)$ is nontrivial. Given graphs G and H, we write $G \stackrel{alg}{\longrightarrow} H$ if $\cl A(G,H)$ is nontirvial.  We define the algebraic chromatic number of G to be
$$\chi_{alg}(G)= min \{c \mid \cl A(G, K_c) \ is \ nontrivial \} $$
\end{defn}

The following is a slight generalization of \cite[Theorem~4.7]{OP2}.

\begin{thm} Let $\cl G=(I,O, \lambda)$ be a synchronous game.
\begin{enumerate}
\item $\cl G$ has a perfect deterministic strategy
 if and only if
there exists a unital *-homomorphism from $\cl A(\cl G)$ to $\bb C$.
\item $\cl G$ has a perfect q-strategy if and only if there exists a
unital *-homomorphism from $\cl A(\cl G)$ to $B(\cl H)$ for some
non-zero finite dimensional Hilbert space.
\item $\cl G$ has a perfect qc-strategy if and only if there exists a
unital C*-algebra $\cl C$ with a faithful trace and a unital *-homomorphism
$\pi: \cl A(\cl G) \to \cl C$.
\end{enumerate}
Hence, if $\cl G$ has a perfect qc-strategy, then it has a perfect algebraic strategy and so $\chi_{qc}(G) \ge \chi_{alg}(G)$ for every graph $G$.
\end{thm}

\begin{proof} We start with the third statement.
Since the game is synchronous, any perfect strategy  $p(a,b|v,w)$ must also be synchronous.
By \cite[Theorem~5.5]{PSSTW}, any synchronous density is of the following form:
$p(a,b|v,w) = \tau(E_{v,a}E_{w,b})$, where $\tau: \cl C \to \bb C$ is a tracial state for a
a unital C*-algebra $\cl C$ generated by projections
$\{ E_{v,a} \}$ satisfying $\sum_a E_{v,a} = I$ for all $v$.


If we take the GNS representation \cite{KR} of $\cl C$ induced by $\tau$,
 then the image of $\cl C$ under this representation will be a
 quotient of $\cl C$ with all the same properties and the additional property
  that $\tau$ is a faithful trace on the quotient.

Now if, in addition, $p(a,b|v,w)$ belongs to the smaller family of perfect q-strategies,
 then by \cite[Theorem~5.3]{PSSTW} the C*-algebra $\cl C$ will be finite dimensional.
Hence, the second statement follows.

Finally, if $p(a,b|v,w)$ belongs to the smaller family of perfect loc-strategies,
then the C*-algebra $\cl C$ will be abelian, and hence, the first statement follows.
\end{proof}

\begin{remark} It is also possible to characterize the existence of perfect qa-strategies,
 but the proof is a bit long for here: $\cl G$ has a perfect qa-strategy
 if and only if there exists a unital *-homomorphism of $\cl A(\cl G)$
  into the von Neumann algebra $\cl R^{\omega}$.
\end{remark}

\ssec{Putting an order on  our *-algebra} 
The *-algebra $\bb C[\bb F(n,m)]$ also possesses an order defined as follows:
let $\cl P$ be the cone generated by all elements of the form $f^* f$ for $f \in \bb C[\bb F(n,m)].$
If $h,k \in \bb C[\bb F(n,m)]$ are self-adjoint elements, we write $h \le k$ if $k-h \in \cl P$.
Next, notice that $\cl P$ induces a cone on $\cl A(\cl G)$,
 which we regard as the positive elements, by setting
\[ \cl A(\cl G)^+ = \{ p + \cl I(\cl G):  p \in \cl P \}.\]
Given two self-adjoint elements $h,k\in\cl A(\cl G)$,
we again write $h \le k$ if and only if $k-h \in \cl A(\cl G)^+$.
In the language of Ozawa~\cite{Oz13} this makes $\cl A(\cl G)$ into a \df{semi-pre-C*-algebra}.

A self-adjoint vector subspace $V \subseteq \bb C[\bb F(n,m)]$ is called \df{hereditary} provided that $0 \le f \le h$ and $h \in V$ implies that $f \in V$.

\begin{prob} Let $\cl G$ be a synchronous game.
Find conditions on the game so that the 2-sided ideal $\cl I(\cl G)$ is hereditary.
\end{prob}

Later we will see an example of a game such that $\cl I(\cl G)$ is not hereditary. The following result shows why the hereditary condition is important.

\begin{prop} Let $\cl G$ be a synchronous game and let $\cl I(\cl G)$ be the ideal of the game.  Then $\cl I(\cl G)$ is a hereditary subspace of $\bb C[\bb F(n,m)]$ if and only if $\big(\cl A(\cl G)^+ \big) \cap \big( - \cl A(\cl G)^+ \big) = (0)$.
\end{prop}
\begin{proof}  Let $x=x^* \in \bb C[\bb F(n,m)]$.  We begin by characterizing when the equivalence class $x+\cl I(\cl G)$ is contained in $ \cl A(\cl G)^+ \cap \big(- \cl A(\cl G)^+ \big)$. By definition, this occurs if and only if there are elements $p=p^*, q= q^*$ in $\cl I(\cl G)$ such that $x+p \ge 0$ and $-x+q \ge 0$. This is equivalent to $0 \le x+p \le p+q$.

Now suppose that $x=x^*$ and that the equivalence class $x+ \cl I(\cl G)$ is non-zero in $\cl A(\cl G)$. If the class is contained in $\big(\cl A(\cl G)^+ \big) \cap \big( - \cl A(\cl G)^+ \big)$ then choosing $p$ and $q$ as in the previous paragraph, the element $x+p$ demonstrates that $\cl I(\cl G)$ is not hereditary.

Conversely, if $\cl I(\cl G)$ is not hereditary, then there exists $x=x^* \notin \cl I(\cl G)$ and $q \in \cl I(\cl G)$ such that $0 \le x \le q$. The inequality $0\le x$ implies that $x+ \cl I(\cl G) \in \cl A(\cl G)^+$, while $0 \le q-x$ implies that
$q-x + \cl I(\cl G) = -x + \cl I(\cl G) \in \cl A(\cl G)^+$.  Clearly, this element is non-zero.
\end{proof}

  Given a subspace $V$ we let $V^h$ denote the smallest hereditary subspace that contains $V$ and call this space the \df{hereditary closure} of $V$.
We define the \df{hereditary *-algebra of the game $\cl G$} to be the quotient
\[ \cl A^h(\cl G) = \bb C[\bb F(n,m)]/\cl I^h(\cl G). \]
Note that $\cl A^h(\cl G)$ is a quotient of $\cl A(\cl G)$.

\begin{defn} We say that a game has a perfect hereditary strategy if $\cl A^h(\cl G)$ is nontrivial. Given graphs G and H, we write $G \stackrel{hered}{\longrightarrow} H$ if $\cl A^h(G,H)$ is nontrivial. We define the hereditary chromatic number of G by
\[ \chi_{hered}(G) = \min \{ c \mid \cl A^h(G, K_c) \ is \ nontrivial \}.\]
\end{defn}


We define the positive cone in $\cl A^h(\cl G)$ by setting
\[ \cl A^h(\cl G)^+ = \{ p+ \cl I^h(\cl G) : p \in \cl P \},\]
so that $\cl A^h(\cl G)$ is also a semi-pre-$C^*$-algebra.
The following is immediate:

\begin{prop}\label{propercone} Let $\cl G$ be a synchronous game, then $$\big(\cl A^h(\cl G)^+ \big) \cap \big( - \cl A^h(\cl G)^+ \big) = (0).$$
\end{prop}
\noindent
Thus, the ``positive'' cone on $\cl A^h(\cl G)$ is now a \df{proper} cone and $\cl A^h(\cl G)$ is an ordered vector space.

\begin{prop}
\label{prop:K_n-->K_c-hered}
If $K_n \stackrel{hered}{\longrightarrow} K_c$ then $n \le c$. Consequently, $\chi_{hered}(K_n) =n.$  \end{prop}
\begin{proof} Suppose, that $K_n \stackrel{hered}{\longrightarrow} K_c$. 
We let $E_{v,i} = e_{v,i} + \cl I^h(K_n,K_c) \in \cl A^h(K_n,K_c)$, where $\cl A^h$ is defined as in \S\ref{sec:alg-sync-game}. Then we have that $\sum_{i=0}^{c-1} E_{v,i} = I$ for all $v.$  Set $P_i = \sum_v E_{v,i}.$  Since  $E_{v,i}E_{w,i} =0,$ we have that $P_i=P_i^* = P_i^2.$  Hence,  $Q_i = I-P_i$ is also a ``projection'' in the sense that $Q_i=Q_i^* = Q_i^2.$ Now $\sum_i P_i = cI - \sum_i Q_i.$  But also,
\[ \sum_i P_i = \sum_v \sum_i E_{v,i} = nI.\]
Hence,  $\sum_i Q_i^2 = \sum_i Q_i = (c-n)I.$  By definition, $\sum_i Q_i^2\in\cl A^h(G,H)^+$. If $c<n$, then $(c-n)I\in -(\cl A^h(G,H)^+)$ and by Proposition \ref{propercone}, we have $I=0$; that is, $1\in\cl I^h(K_n,K_c)$ which contradicts our hypothesis that $K_n \stackrel{hered}{\longrightarrow} K_c$. Hence, $c\ge n$. 
This shows that $\chi_{hered}(K_n)\ge n$.

For the other inequality, note that if $K_n \stackrel{hered}{\longrightarrow} K_c$, then $\chi_{hered}(K_n) \le c$.
By the results of \cite{OP2}, if  $G$ is any graph with $c= \chi_{qc}(G)$, then there is a unital *-homomorphism from $\cl A(G, K_c)$ into a C*-algebra with a trace. The kernel of this homomorphism is a hereditary ideal and so must contain $\cl I^h(G,K_c)$. Hence, this latter ideal is proper and $c \ge \chi_{hered}(G)$.
Thus,  $\chi_{hered}(K_n) \le \chi_{qc}(K_n) \le \chi(K_n) =n$. Hence, $n= \chi_{hered}(K_n) \le c$.
\end{proof}

The above proof also shows that:
\begin{prop} If $K_n \stackrel{alg}{\to} K_c$ and $-I \notin \cl A(K_n,K_c)^+$, then $n \le c$.
\end{prop}

\ssec{Pre-$C^*$-algebras}
 The next natural question is whether or not $\cl A^h(\cl G)$  is a \df{pre-$C^*$-algebra} in the sense of \cite{Oz13}. The answer is that this cone will need to satisfy one more hypothesis.

\begin{defn} Let $\cl G$ be a synchronous game and let $\cl I^c(\cl G)$ \index{$\cl I^c(\cl G)$}
 denote the intersection of the kernels of all unital $*$-homomorphisms
 from $\bb C[\bb F(n,m)]$ into the bounded operators on a Hilbert space (possibly 0 dimensional) that vanish on $\cl I(\cl G)$. Let $\cl A^c(\cl G) = \bb C[\bb F(n,m)]/ \cl I^c(\cl G)$.
\end{defn}

\begin{prop} Let $\cl G$ be a synchronous game.  Then  $\cl I^h(\cl G) \subseteq \cl I^c(\cl G)$ and
\[ \cl I^c(\cl G) = \{ x \in \bb C[\bb F(n,m)] \, : x^*x + \cl I(\cl G) \le \epsilon 1 + \cl I(\cl G), \, \forall \epsilon > 0, \, \epsilon \in \bb R \}.\]
There exists a (non-zero) Hilbert space $\cl H$ and a unital *-homomorphism $\pi: \bb C[\bb F(n,m)] \to B(\cl H)$ that vanishes on $\cl I(\cl G)$ if and only if $\cl A^c(\cl G) \ne (0)$.
\end{prop}
\begin{proof} The kernel of every *-homomorphism is a hereditary ideal and the intersection of hereditary ideals is a hereditary ideal, hence $\cl I^c(\cl G)$ is a hereditary ideal containing $\cl I(\cl G)$. So, $\cl I^h(\cl G) \subseteq \cl I^c(\cl G)$.

  We have that $x \in \cl I^c(\cl G)$ if and only if $x + \cl I(\cl G)$ is in the kernel of every *-homomorphism of $\cl A(\cl G)$ into the bounded operators on a Hilbert space. In \cite[Theorem~1]{Oz13} it is shown that this is equivalent to $x+ \cl I(\cl G)$ being in the ``ideal of infinitesimal elements'' of $\cl A(\cl G)$, which is the ideal defined by the right-hand side of the above formula. The last result comes from the fact that the ideal of infinitesimal elements is exactly the intersection of the kernels of all such representations.
\end{proof}

\begin{defn} We say that a game $\cl G$ has a {\bf perfect C*-strategy} provided that $\cl A^c(\cl G)$ is nontrivial. Following \cite{OP2}, we write $G \stackrel{C^*}{\longrightarrow} H$ provided that for given graphs $G$ and $H$, the algebra $\cl A^c(G,H)$ is nontrivial. We define the C*-chromatic number of G to be
\[ \chi_{C^*}(G) = \min \{ c \mid \cl A^c(G, K_c) \ is \ nontrivial \}.\]

\end{defn}

The following is immediate.

\begin{prop} Let $G$ be a graph. Then $\chi_{qc}(G) \ge \chi_{C^*}(G) \ge \chi_{hered}(G) \ge \chi_{alg}(G)$.
\end{prop}

This motivates the following question.

\begin{prob} Let $\cl G$ be a synchronous game. Is $\cl I^c(\cl G) = \cl I^h(\cl G)$?
\end{prob}

\begin{prob} If $\cl I^h(\cl G) \ne \bb C[\bb F(n,m)]$, then does there exist
a non-zero Hilbert space $\cl H$ and a unital *-homomorphism,
$\pi: \bb C[\bb F(n,m)] \to B(\cl H)$
such that $\pi(\cl I^h(\cl G)) = (0)$, that is, if $\cl I^h(\cl G) \ne \bb C[\bb F(n,m)]$,
then is $\cl I^c(\cl G) \ne \bb C[\bb F(n,m)]$?
\end{prob}

\ssec{Determining if an ideal is hereditary}
\label{sec:computingHered}
Here we  mention some literature on determining if an ideal ${\cl I}$ is
hereditary and the issue of  computing its ``hereditary closure.''
In the real algebraic geometry literature, a hereditary ideal is
called a real ideal.
For a finitely generated left ideal $\cl I$ in ${\bb R}(\bb F(k))$
the papers \cite{CHKM13, CHKMN14} present a theory and a numerical algorithm
to test (up to numerical error) if  $\cl I$ is hereditary.
The algorithm also computes the ``hereditary radical'' of $\cl I$.
The computer algorithm relies on numerical optimization (semidefinite programming) and
hence it is not exact but approximate.

For two sided ideals \cite{CHMN15}  and \cite{KVV}  contain some theory.
Also the first author and Klep developed and crudely implemented a hereditary 
ideal algorithm under NCAlgebra. However, it is too memory consuming to be
effective, so we leave this topic for future work.

A moral one can draw from this literature is that computing hereditary closures
is not broadly effective at this moment.

\subsection{Clique Numbers}

 The clique number of a graph $\omega(G)$ is defined as the size of the largest complete subgraph of G. It is not hard to see that G contains a complete subgraph of size c if and only if there is a graph homomorphism from $K_c$ to G. Hence, there is a parallel theory of quantum clique numbers that we shall not pursue here, other than to remark that for each of the cases $t \in \{ loc, q, qa, qc, C^*, hered, alg \}$ we define the {\bf t-clique number} of g by
 \[ \omega_t(G) = \max \{ c \mid K_c \stackrel{t}{\longrightarrow} G \},\]
 so that
 \begin{multline*} \omega(G)=\omega_{loc}(G) \le \omega_q(G) \le \omega_{qa}(G) \\ \le \omega_{qc}(G) \le \omega_{C^*}(G) \le \omega_{hered}(G) \le \omega_{alg}(G).\end{multline*}

Lovasz \cite{Lo} introduced his theta function $\vartheta (G)$ of a graph. The famous Lovasz sandwich theorem \cite{GR} says that for every graph $G$, if $\overline{G}$ denotes its graph complement, then
$\omega(G) \le \vartheta (\overline{G}) \le \chi(G)$.
In \cite[Proposition~4.2]{OP2} \label{thetainequality} they showed the following improvement of the Lovasz sandwich theorem:
 \[ \omega_{C^*}(G) \le \vartheta (\overline{G}) \le \chi_{C^*}(G).\]
 We shall show later, $\chi_{alg}(K_5) =4$, while $\vartheta(\overline{K_5})=5.$
Hence the sandwich inequality fails for the algebraic version.

This motivates the following problem:
\begin{prob} Is $\omega_{hered}(G) \le \vartheta(\overline{G}) \le \chi_{hered}(G)$ for all graphs?
\end{prob}


\section{The case of 1, 2 and  3 colors}
\label{sec:123colors}

It is a classic result that deciding if $\chi(G) \le 3$ is an NP-complete problem. In \cite{Ji} it was shown that deciding if $\chi_{q}(G) \le 3$ is NP-hard, and, in particular, there is no known algorithm for deciding if this latter inequality is true.
For these reasons it is interesting to see what can be said about the new inequalities, $\chi_{C^*}(G) \le 3, \,\, \chi_{hered}(G) \le 3,$ and $\chi_{alg}(G) \le 3$.  Addressing the first two inequalities would require one to compute $\cl I^c(G, K_3)$ and $\cl I^h(G, K_3)$, and unfortunately these ideals contain elements not just determined by simple algebraic relations.
However, studying $\cl A(G, K_2)$  and $\cl A(G, K_3)$
is rewarding, as we shall see now see.
 Throughout the section, we use the notation $E_{v,i} = e_{v,i} + \cl I(K_n,K_c)$.

\begin{prop} Let $G$ be a graph. Then $\chi_{alg}(G)=1$ if and only if $G$ is the empty graph. Hence, $\chi_{alg}(G)=1 \iff \chi(G)=1.$
\end{prop}
\begin{proof} For each vertex we only have one idempotent $E_{v,1}$ and since these sum to the identity, necessarily $E_{v,1}=I$. But if there is an edge $(v,w)$ then $I=I \cdot I = E_{v,1}E_{w,1}=0$.
\end{proof}

\begin{prop}\label{prop:alg=2}  Let $G$ be a connected graph on more than one vertex.  Then $\chi_{alg}(G)=2 \iff \chi(G)=2.$
\end{prop}
\begin{proof}  First assume that $\chi_{alg}(G) =2$. Fix a vertex $v$ and set $P_0 = E_{v,0}$ and $P_1 = E_{v,1}.$  Note $P_0+ P_1 = I.$ Let $(v,w) \in E(G)$, then  $P_0E_{w,0} =0$ and $P_1E_{w,1} =0.$ Hence,  $E_{w,0}= (P_0+P_1)E_{w,0} = P_1E_{w,0}$ and similarly, $E_{w,0} = E_{w,0}P_1.$ Also, $P_1= P_1(E_{w,0} + E_{w,1}) = P_1E_{w,0}= E_{w,0}.$

Thus, whenever $(v,w) \in E(G),$ then $E_{v,i} = E_{w,i+1},$ i.e., there are two projections and they flip. Since $G$ is connected, by using a path from $v$ to an arbitrary $w$ we see that $\{ E_{w,0}, E_{w,1} \} = \{ P_0, P_1 \}.$

Now we wish to 2-color $G$. Define the color of any vertex $w$ to be $0$ if $E_{w,0} = P_0$ and 1 if $E_{w,0} = P_1.$ This yields a 2-coloring, and since $G$ is connected on more than one vertex, there is no 1-coloring, showing $\chi(G)=2$.

Conversely, if $\chi(G)=2$ then $G$ is not the empty graph. Since $1 \le \chi_{alg}(G) \le 2$, by the previous result, $\chi_{alg}(G)=2$.
\end{proof}

\begin{prop}
\label{prop:G-K-3-commute}
If $(v,w) \in E(G)$ then $E_{v,i}E_{w,j} = E_{w,j}E_{v,i} \in \cl A(G, K_3)$ for all $i,j$. In particular, if $G$ is complete, then $\cl A(G, K_3)$ is abelian.
\end{prop}
\begin{proof} For $0= E_{v,0}E_{v,1} = E_{v,0}(E_{w,0} + E_{w,1} + E_{w,2}) E_{v,1} = E_{v,0}E_{w,2}E_{v,1}.$  Similarly,  $E_{v,i}E_{w,j}E_{v,k}=0$ whenever $\{i,j,k\} = \{ 0,1,2 \}.$

Now  $E_{w,0} = (E_{v,0} +E_{v,1} + E_{v,2})E_{w,0}(E_{v,0}+E_{v,1} +E_{v,2}) = E_{v,1}E_{w,0}E_{v,1} + E_{v,2}E_{w,0}E_{v,2}.$  Similarly,  $E_{w,j} = E_{v,j+1}E_{w,j}E_{v,j+1} + E_{v,j+2}E_{w,j}E_{v,j+2}.$

Hence, for $i \ne j$,  $E_{v,i}E_{w,j} = E_{v,i}E_{w,j}E_{v,i} = E_{w,j} E_{v,i}$, while when $i=j$,  $E_{v,i}E_{w,i} = 0= E_{w,i}E_{v,i}.$
\end{proof}

\begin{thm}  $\chi_{alg}(K_j) =j$ for $j=2,3,4.$
\end{thm}
\begin{proof}  We have that $\chi_{alg}(K_2) =2,$ by Proposition~\ref{prop:alg=2}.  Now if $\chi_{alg}(K_3) =2$ then by Proposition \ref{prop:alg=2}, we see $\chi(K_3) = 2,$ which is a contradiction.  Hence,  $3 \le \chi_{alg}(K_3) \le \chi(K_3) =3.$

Finally, if $\chi_{alg}(K_4) =3$, then by Proposition \ref{prop:G-K-3-commute}, we have that $\cl A(K_4,K_3)$ is a non-zero abelian complex *-algebra.
  But every unital, abelian ring contains a proper maximal ideal $M$, and forming the quotient we obtain a field $\bb F$. The map $\lambda 1 \to \lambda 1 +M$ embeds $\bb C$ as a subfield. Now we use the fact that $\cl A(K_4, K_3)$ is generated by projections and that the image of each projection in $\bb F$ is either 0 or 1 in order to see that the range of the quotient map is just $\bb C$. Thus, $\bb F= \bb C$ and we have a unital homomorphism $\pi: \cl A(K_4,K_3) \to \bb C$ and again using the fact that the image of each projection is 0 or 1 and that the projections commute, we see that $\pi$ is a *-homomorphism.  Hence, by \cite[Theorem~4.12]{OP2}, we have  $\chi(K_4) \le 3$, a contradiction.

Thus, $3< \chi_{alg}(K_4) \le \chi(K_4)=4$ and the result follows.
\end{proof}

\begin{cor} $I \in \cl I(K_4,K_3)$. \end{cor}

\section{*-Algebra versus Free Algebra}

\label{sec:alg}

The original motivation for the construction of the algebra of a game comes
from projective quantum measurement systems which are always given
by orthogonal projections on a Hilbert space, i.e., operators satisfying $E=E^2=E^*$.
This is why we have defined the algebra of a game to be a *-algebra.
But a natural question is whether or not one really needs a *-algebra
or is there simply a free algebra with relations that suffices?
In this section we show that as long as one introduces the correct relations
then the assumption that the algebra be a *-algebra is not necessary.

To this end let $\cl F(nm) :=\bb C \< x_{v,a}\mid 0 \le v \le n-1,  0 \le a \le m-1 \>$ be the free unital complex algebra on $nm$ generators and let  $\cl B(n,m)= \cl F(nm)/\cl I(n,m)$
 where $\cl I(n,m)$ is the two-sided ideal generated by
\[ x_{v,a}^2 - x_{v,a}, \forall v,a ; \quad 1 - \sum_{a=0}^{m-1} x_{v,a}, \forall v;
\quad  x_{v,a}x_{v,b}, \forall v, \forall a \ne b.\]
We let $p_{v,a}$ denote the coset of $x_{v,a}$ in the quotient so that
\[ p_{v,a}^2 = p_{v,a}, \forall v,a;
\quad  1 = \sum_{a=0}^{m-1} p_{v,a}, \forall v; \quad  p_{v,a}p_{v,b}=0, \forall v, \forall a \ne b.\]

\begin{prop}
There is an isomorphism $\pi: \cl B(n,m) \to \bb C[\bb F(n,m)]$ with $\pi(p_{v,a}) = e_{v,a}, \forall v, a$, where $e_{v,a}$ are defined as in the previous section.
\end{prop}
\begin{proof}
Let $\rho: \cl F(nm) \to \bb C[\bb F(n,m)]$ be
the unital algebra homomorphism 
with $\rho(x_{v,a}) = e_{v,a}$
Then $\rho$ vanishes on $\cl I(n,m)$ and so induces a quotient homomorphism $\pi: \cl B(n,m) \to \bb C[\bb F(n,m)]$.  It remains to show that $\pi$ is one-to-one.

To this end set $\omega= e^{2 \pi i/m}$ and let $y_v = \sum_{a=0}^{m-1} \omega^a p_{v,a}$. It is readily checked that $y_{v}$  $y_v^{m} = \sum_{a=0}^{m-1} \omega^{-am} p_{v,a}=1$.  Since $p_{v,a} = \frac{1}{m} \sum_{k=0}^{m-1} (\omega^{-a} y_v)^k$ we have that $\{ y_v: 0 \le v \le n-1 \}$ generates $\cl B(n,m)$.

Now by the universal property of $\bb C[\bb F(n,m)]$ there is a homomorphism $\gamma: \bb C[\bb F(n,m)] \to \cl B(n,m)$ with $\gamma(u_v) = y_v$ and hence, this is the inverse of $\pi$.
\end{proof}

If a unital algebra contains 3 idempotents, $p_1, p_2, p_3$ with $p_1+p_2+p_3 =1$, then necessarily $p_ip_j =0$ for $i \ne j$. To see this note that $p_1+p_2 = 1 - p_3$ is idempotent. Squaring yields that $p_1p_2+p_2p_1=0$. Thus, $0= p_1(p_1p_2+p_2p_1) = p_1p_2 + p_1p_2p_1$ and $0 = (p_1p_2+p_2p_1)p_1 = p_1p_2p_1 + p_2p_1$, from which it follows that $p_1p_2=p_2p_1$ and so $2p_1p_2 =0$ and the claim follows.

Hence, when $m=3$ the condition that $x_{v,a}x_{v,b} =0, \forall a \ne b$ is a consequence of the other hypotheses and is not needed in the definition of the ideal $\cl I(n,m)$.

However, Heydar Radjavi \cite{Ra} has constructed a set of 4 idempotents in a unital complex algebra which sum to the identity but for which $p_ip_j \ne 0$ for $i \ne j$.
Thus, for $m \ge 4$, it is necessary to include the relation $x_{v,a}x_{v,b}$ in the ideal in order to guarantee that the quotient $\cl B(n,m)$ is isomorphic to $\bb C[\bb F(n,m)]$, since these products are 0 in the latter algebra.

If a set of self-adjoint projections on a Hilbert space, $P_1,...,P_m$ sum to the identity then it is easily checked that they project onto orthogonal subspaces and so $P_iP_j =0, \forall i \ne j$.
Thus, in any C*-algebra when self-adjoint idempotents sum to the identity,
their pairwise products are 0.  But the situation is not so clear for
self-adjoint idempotents in a *-algebra and we have not been able to resolve this question.
So we ask:

\begin{prob} Let $\cl A$ be a unital *-algebra and let $p_1,...,p_m$ satisfy $p_i=p_i^2=p_i^*$ and $p_1+ \cdots + p_m =1$. Then does it follow that $p_ip_j=0, \forall i \ne j$?
\end{prob}

\begin{cor} Let $\cl G=(I,O, \lambda)$ be a symmetric synchronous game with $|I|=n$ and $|O|=m$.  Then $\cl A(\cl G)$ is isomorphic to the quotient of $\cl F(nm)$
by the 2-sided ideal generated by
\[ x_{v,a}^2- x_{v,a}, \forall v,a;  \quad 1 - \sum_{a=0}^{m-1} x _{v,a}, \forall v \]
and
\[ x_{v,a} x_{w,b}, \forall v,w,a,b \text{ such that } \lambda(v,w,a,b) =0.\]
\end{cor}
Note that $x_{v,a}x_{v,b}$ for $a \ne b$ is in the ideal since $\lambda(v,v,a,b)=0$.
We note that the hypothesis that the game be symmetric is needed, since in a *-algebra, the condition that $x_{v,a}x_{w,b} =0$ implies that $x_{w,b}x_{v,a}=0$, while this relation would not necessarily be met in the quotient of the free algebra.

\begin{cor}
\label{cor:1notin-I-Fnm}
A symmetric synchronous game $\cl G= (I,O, \lambda)$
 has a perfect algebraic strategy if and only if 1 is not in the 2-sided ideal of $\cl F(nm)$ generated by
\[ x_{v,a}^2- x_{v,a}, \forall v,a; \quad  1 - \sum_{a=0}^{m-1} x _{v,a}, \forall v,\]
and
\[ x_{v,a} x_{w,b}, \forall v,w,a,b \text{ such that } \lambda(v,w,a,b) =0.\]
\end{cor}

\subsection{Change of Field}
\label{sec:field}
The following is important for  Gr\"obner basis calculations.  Let  $1 \in \bb K \subseteq \bb C$ be any subfield.  We set $\cl F_{\bb K}(nm)$ equal to the free $\bb K$-algebra on $nm$ generators $x_{v,i}$, so that $\cl F_{\bb C}(nm) = \cl F(nm)$.  Given any symmetric synchronous game $\cl G$ with $n$ inputs and $m$ outputs, we let $\cl I_{\bb K}(\cl G) \subseteq \cl F_{\bb K}(nm)$ be the 2-sided ideal generated by
\[ x_{v,a}^2- x_{v,a}, \forall v,a;  \quad  1 - \sum_{a=0}^{m-1} x _{v,a}, \forall v, \]
and
\[ x_{v,a} x_{w,b}, \forall v,w,a,b \text{ such that } \lambda(v,w,a,b) =0.\]
We let $\cl A_{\bb K}(\cl G) = \cl F_{\bb K}(nm)/\cl I_{\bb K}(\cl G)$.
 By Corollary \ref{cor:1notin-I-Fnm}, $\cl G$ has a perfect algebraic strategy
 if and only if $1 \not \in \cl I_{\bb C}(\cl G)$, or equivalently $\cl A_{\bb C}(\cl G)\neq0$.

We show that this computation is independent of the field $\bb K$.
\begin{prop}
\label{prop:A_K-base-change}
If $\bb K$ is a field containing $\bb Q$, then
$$\cl A_{\bb K}(\cl G)=\cl A_{\bb Q}(\cl G)\otimes_{\bb Q}\bb K.$$
Furthermore, $\cl A_{\bb K}(\cl G)=0$ if and only if $\cl A_{\bb Q}(\cl G)=0$.
\end{prop}
\begin{proof}
By definition, we have a short exact sequence
\[
0 \to \cl I_{\bb Q} \to \cl F_{\bb Q} \to \cl A_{\bb Q} \to 0
\]
of $\bb Q$-vector spaces. Since $\bb K$ is flat over $\bb Q$, we obtain a short exact sequence
\[
0 \to \cl I_{\bb Q}\otimes_{\bb Q}\bb K \to \cl F_{\bb Q}\otimes_{\bb Q}\bb K \to \cl A_{\bb Q}\otimes_{\bb Q}\bb K \to 0.
\]
Since the generators are independent of the field, one checks that $\cl F_{\bb Q}\otimes_{\bb Q}\bb K=\cl F_{\bb K}$ and that the image of $\cl I_{\bb Q}\otimes_{\bb Q}\bb K \to \cl F_{\bb K}$ is equal to $\cl I_{\bb K}$. Since this latter map is injective, we see $\cl I_{\bb Q}\otimes_{\bb Q}\bb K=\cl I_{\bb K}$. Hence, the above short exact sequence shows $\cl A_{\bb K}(\cl G)=\cl A_{\bb Q}(\cl G)\otimes_{\bb Q}\bb K$.

Lastly, $\bb K/\bb Q$ is faithfully flat. So, $\cl A_{\bb K}(\cl G)=0$ if and only if $\cl A_{\bb Q}(\cl G)=0$.
\end{proof}

\section{The case of 4 colors}
\label{sec:4colors}

\def\vari{x}

 This section gives a machine-assisted proof which analyzes 4 algebraic colors. We prove:
 
\begin{thm}
	\label{thm:4colors}
	For any graph $G$, we have $\chi_{alg}(G)\leq 4$.
\end{thm}

This theorem is equivalent to the statement that for any  graph $G$, the ideal satisfies $1\not\in\mathcal{I}(G,K_4)$. We will prove this statement through the use of (noncommutative) Gr\"obner bases. For a brief effective exposition to noncommutative Gr\"obner basis algorithms, see Chapter 12.3 \cite{F97} or \cite{Mor94, Rei95, Lev05}.

For those readers already familiar with (commutative) Gr\"obner bases, we explain the key differences with the noncommutative setting. Let $\cl I = (p_1, \dots p_k)$ be a two-sided ideal, and prescribe a monomial order. A noncommutative Gr\"obner basis $\cl B$ of $\cl I$ is a set of generators such that the leading term of any element of $\cl I$ is in the monomial ideal generated the leading terms of $\cl B$. A noncommutative Gr\"obner basis is produced in the same way as in the commutative case. Let $m_j$ be the leading term of $p_j$ and notice that any two $m_j, m_k$ have as many as 4 possible least common multiples, each of which produces syzygyies. One repeatedly produces syzygyies and reduces to obtain a Gr\"obner basis in the same way as the commutative setting. However, unlike the commutative case, a Gr\"obner basis can be infinite. Very fortunately the Gr\"obner bases that arise in our coloring computations below are finite. The key property we use is that $p$ is in $\cl I$ if and only if the reduction of $p$ by a Gr\"obner basis for $\cl I$ yields 0.

Recall that $\mathcal{I}(G,K_4)$ is generated by the following relations:
 $$\vari_{v,i}\vari_{v,j}\ \  \forall i\neq j;\ \ 1-\sum_{i=0}^{3}\vari_{v,i}\ \   \forall v; \ \ \vari_{v,i}\vari_{w,i} \ \ \forall (v,w)\in E(G), \, \forall i.$$

 To prove Theorem \ref{thm:4colors} we will make use of the following theorem:

\begin{thm}
\label{Grob:Bas}
For any $n\geq 3$ a Gr\"obner basis for $\mathcal{I}(K_n, K_4)$ under the graded lexographic ordering with
\beq
\label{eq:order}
\vari_{0,0}<\vari_{0,1}<\vari_{0,2}<\vari_{0,3}<\vari_{1,0}<\vari_{1,1}<\ldots < \vari_{n-1,3}
\eeq
consists of relations of the following forms:

\ben
	 \item
	 	\label{it:1}
	 $$\vari_{v,i}\vari_{v,j}$$ with $i,j\leq 2$  $i\neq j$

	\item
		\label{it:2}
	$$\vari_{v,i}^2-\vari_{v,i}$$ with $i\leq 2$

	\item
		\label{it:3}
	$$\vari_{v,3}+\vari_{v,2}+\vari_{v,1}+\vari_{v,0}-1$$

	\item
		\label{it:4}
	$$\vari_{v,i}\vari_{w,i}$$ with  $v\neq w$, $i\leq 3$

	\item
		\label{it:5}
	$$\vari_{v,2}\vari_{w,1}  + \vari_{v,2}\vari_{w,0}+\vari_{v,1}\vari_{w,2}+\vari_{v,1}\vari_{w,0}+\vari_{v,0}\vari_{w,2}+\vari_{v,0}\vari_{w,1}$$ $$-\vari_{v,2}-\vari_{v,1}-\vari_{v,0}-\vari_{w,2}-\vari_{w,1}-\vari_{w,0}+1$$  with $v\neq w$

	\item
		\label{it:6}
	$$\vari_{v,2}\vari_{w,0}\vari_{v,1}  - \vari_{v,1}\vari_{w,2}\vari_{v,0}-\vari_{v,1}\vari_{w,0}\vari_{v,2}-\vari_{v,0}\vari_{w,2}\vari_{v,0}-\vari_{v,0}\vari_{w,1}\vari_{v,2}-\vari_{v,0}\vari_{w,1}\vari_{v,0}$$ $$+\vari_{v,1}\vari_{w,2}+\vari_{v,1}\vari_{w,0}+\vari_{v,0}\vari_{w,2}+\vari_{v,0}\vari_{w,1}+\vari_{w,2}\vari_{v,0}+\vari_{w,1}\vari_{v,2}+\vari_{w,1}\vari_{v,0}+\vari_{w,0}\vari_{v,2}$$ $$-\vari_{v,2}-\vari_{v,1}-\vari_{v,0}-\vari_{w,2}-\vari_{w,1}-\vari_{w,0}+1$$ with  $v\neq w$

	\item
	\label{it:7}
	$$\vari_{v,2}\vari_{w,0}\vari_{x,1}  -  \vari_{v,1}\vari_{w,2}\vari_{x,0}-\vari_{v,1}\vari_{w,0}\vari_{x,2}-\vari_{v,0}\vari_{w,2}\vari_{x,0}-\vari_{v,0}\vari_{w,1}\vari_{x,2}-\vari_{v,0}\vari_{w,1}\vari_{x,0}$$ $$
+\vari_{v,2}\vari_{x,0}+
\vari_{v,1}\vari_{w,2}+\vari_{v,1}\vari_{w,0}+2\vari_{v,1}\vari_{x,2}+2\vari_{v,1}\vari_{x,0}+
\vari_{v,0}\vari_{w,2}+\vari_{v,0}\vari_{w,1}+2\vari_{v,0}\vari_{x,2}$$ $$+\vari_{v,0}\vari_{x,1}
+\vari_{w,2}\vari_{x,0}+\vari_{w,1}\vari_{x,2}+\vari_{w,1}\vari_{x,0}+\vari_{w,0}\vari_{x,2}$$ $$
-\vari_{v,2}-2\vari_{v,1}-2\vari_{v,0}-\vari_{w,2}-\vari_{w,1}-\vari_{w,0}-2\vari_{x,2}-\vari_{x,1}-2\vari_{x,0}+2$$ with  $v\neq w\neq x\neq v$

\een

	Specifically $\mathcal{I}(K_n,K_4)$ has a Gr\"obner basis that does not
	contain 1 and thus $1\not\in \mathcal{I}(K_n,K_4)$.

\begin{remark}
\label{rem:GBorder}
Each of the relations \eqref{it:1}--\eqref{it:7} correspond to a set of relations obtained by taking all choices of
 $v,w,x$ in $V(K_n)$. However because of the monomial ordering chosen, the leading terms are always of the forms:


\eqref{it:1} $\vari_{v,i}\vari_{v,j}$
\qquad
\eqref{it:2} $\vari_{v,i}^2$
\qquad
\eqref{it:3} $\vari_{v,3}$
\qquad
\eqref{it:4} $\vari_{v,i}\vari_{w,j}$
\qquad

\eqref{it:5} $\vari_{v,2}\vari_{w,1}$
\qquad
\eqref{it:6}
	$\vari_{v,2}\vari_{w,0}\vari_{v,1}$
\qquad
\eqref{it:7} $\vari_{v,2}\vari_{w,0}\vari_{x,1} $

\bs

Additionally for every $1\leq i \leq 7$, all the vertices of $K_n$ which appear in the
terms of relation (i)  also appear in the
 leading term of (i).
\end{remark}

	\begin{proof}
	Let the ideal generated by these relations be denoted by $\mathcal{J}$, we will first show that these relations form a Gr\"obner basis for $\mathcal{J}$, and then show that $\mathcal{J}=\mathcal{I}(K_n,K_4)$.

Before we begin our calculations pertaining to an algebra over $\bb C$
we note that all of the coefficients that appear will be in $\mathbb{Q}$.
 Section \ref{sec:field} bears on this.

To see that these relations form a Gr\"obner basis we must show that the syzygy between any two polynomials in this list is zero when reduced by the list.
First by Remark \ref{rem:GBorder} each of the relations has variables corresponding to at most three different vertices of $K_n$ and reducing by a relation will not introduce variables corresponding to different vertices.
Thus when calculating and reducing the syzygy between any two relations, variables corresponding to at most 6 vertices of $K_n$ will be involved.	
Therefore we can verify that all syzygies reduce to zero
by looking at the case $n=6$ which we verify using
NCAlgebra 5.0 and NCGB running under Mathematica
(see notebook QCGB-9-20-16.nb, available at:  https://github.com/NCAlgebra/UserNotebooks  ).
This proves that the relations
\eqref{it:1} -- \eqref{it:7}  form a Gr\"obner basis.

We now show that $\mathcal{J}=\mathcal{I}(K_n,K_4)$.
We will first show that all of the generators of $\mathcal{J}$ are contained in $\mathcal{I}(K_n,K_4)$.
		 The elements of types (1), (3), and (4) are self-evidently in $\mathcal{I}(K_n,K_4)$ since they are elements of the generating set of $\mathcal{I}(K_n,K_4)$.
		 For type (2) we note that under the relations generating $\mathcal{I}(K_n,K_4)$ that 	
			$$\vari_{v,i}(1-\sum_{j=0}^3 \vari_{v,j})=\vari_{v,i}-\vari_{v,i}^2-\sum_{j\neq i} \vari_{v,i}\vari_{v,j}=\vari_{v,i}-\vari_{v,i}^2,$$
		and thus elements of type (2) are in $\mathcal{I}(K_n,K_4)$.
		 For type (5) we use the relations generating $\mathcal{I}(K_n,K_4)$ to get that
			$$\vari_{v,3}\vari_{w,3}=(1-\vari_{v,2}-\vari_{v,1}
			-\vari_{v,0})(1-\vari_{w,2}-\vari_{w,1}-\vari_{w,0})$$ 
			$$=\vari_{v,2}\vari_{w,1}+\vari_{v,2}\vari_{w,0}+\vari_{v,1}\vari_{w,2}+\vari_{v,1}\vari_{w,0}+\vari_{v,0}\vari_{w,2}$$   $$+\vari_{v,0}\vari_{w,1}-\vari_{v,2}-\vari_{v,1}-\vari_{v,0}-\vari_{w,2}-\vari_{w,1}-\vari_{w,0}+1.
			$$  
			Finally type (6) is obtained by reducing
			$$(\vari_{v,2}\vari_{w,1}+\vari_{v,2}\vari_{w,0}+\vari_{v,1}\vari_{w,2}
			+\vari_{v,1}\vari_{w,0}+\vari_{v,0}\vari_{w,2}+\vari_{v,0}\vari_{w,1}$$ 
			$$-\vari_{v,2}-\vari_{v,1}-\vari_{v,0}-\vari_{w,2}-\vari_{w,1}
			-\vari_{w,0}+1)\vari_{v,1}-\vari_{v,2}(\vari_{w,1}\vari_{v,1})
			$$
		using the relations of types (1)--(5), and type (7) is obtained by reducing
			$$(\vari_{v,2}\vari_{w,1}+\vari_{v,2}\vari_{w,0}+\vari_{v,1}\vari_{w,2}+\vari_{v,1}\vari_{w,0}+\vari_{v,0}\vari_{w,2}+\vari_{v,0}\vari_{w,1}$$ $$-\vari_{v,2}-\vari_{v,1}-\vari_{v,0}-\vari_{w,2}-\vari_{w,1}-\vari_{w,0}+1)\vari_{x,1}-\vari_{v,2}(\vari_{w,1}\vari_{x,1})
			$$
		using the relations of types (1)--(5).  These two reductions are verified with Mathematica in QCGB-9-20-16.nb.  Thus all of the generating relations of $\mathcal{J}$ are in $\mathcal{I}(K_n,K_4)$ and we have that $\mathcal{J}\subset \mathcal{I}(K_n,K_4)$.

		Next we will show that all the generators of $\mathcal{I}(K_n,K_4)$ are contained in $\mathcal{J}$.
 The only generating relations of $I$ that are not immediately seen to be in $\mathcal{J}$ are $$\vari_{v,3}\vari_{v,j}, \vari_{v,i}\vari_{v,3},$$ and $$\vari_{v,3}\vari_{w,3}.$$
  To see that $\vari_{v,i}\vari_{v,3}$ is in $\mathcal{J}$ we consider
  $\vari_{v,i}(\vari_{v,3}+\vari_{v,2}+\vari_{v,1}+\vari_{v,0}-1)$.
   This is an element of $\mathcal{J}$ since $(\vari_{v,3}+\vari_{v,2}+\vari_{v,1}+\vari_{v,0}-1)$ is in $\mathcal{J}$,
   and when multiplied out all terms except $\vari_{v,i}\vari_{v,3}$ are in $J$,
   and thus $\vari_{v,i}\vari_{v,3}$ is in $\mathcal{J}$, similarly $\vari_{v,3}\vari_{v,j}$ is in $\mathcal{J}$.
    Finally, we consider the equation
			$$\vari_{v,3}\vari_{w,3}=(\vari_{v,3}+\vari_{v,2}+\vari_{v,1}+\vari_{v,0}-1)(\vari_{w,3}+\vari_{w,2}+\vari_{w,1}+\vari_{w,0}-1)$$ $$-(\vari_{v,2}\vari_{w,1}+\vari_{v,2}\vari_{w,0}+\vari_{v,1}\vari_{w,2}+\vari_{v,1}\vari_{w,0}+\vari_{v,0}\vari_{w,2}+\vari_{v,0}\vari_{w,1}$$ $$-\vari_{v,2}-\vari_{v,1}-\vari_{v,0}-\vari_{w,2}-\vari_{w,1}-\vari_{w,0}+1)-\vari_{v,2}\vari_{w,2}-\vari_{v,1}\vari_{w,1}-\vari_{v,0}\vari_{w,0},
			$$
the right-hand side is a sum of relations in $\mathcal{J}$ and is thus in $\mathcal{J}$,
and thus the left-hand side is also in $\mathcal{J}$, specifically $\vari_{v,3}\vari_{w,3}$ is in $\mathcal{J}$.  Therefore all of the generating relations of $\mathcal{I}(K_n,K_4)$ are in $\mathcal{J}$, so that $\mathcal{I}(K_n,K_4)\subset \mathcal{J}$.
 Since we have shown inclusion both ways,
 we have that $\mathcal{I}(K_n,K_4)=\mathcal{J}$ and we are done.
	\end{proof}
\end{thm}

\begin{lemma}
	\label{lem:SubLem}
	If $G$, $H$ are graphs such that $V(H)=V(G)$ and $E(H)\supset E(G)$,
	then  $\mathcal{I}(H,K_m)\supset \mathcal{I}(G,K_m)$ and
	thus $1\not\in \mathcal{I}(H,K_m)\implies 1\not\in\mathcal{I}(G,K_m)$.
\end{lemma}

	\begin{proof}
	The relations generating $\mathcal{I}(H,K_m)$ contains the
	relations generating $\mathcal{I}(G,K_m)$ 
	 and thus the result follows.
	\end{proof}

	\begin{proof}[Proof of Theorem \ref{thm:4colors}]
		 Let $G$ be a graph on $n$ vertices.  By Theorem \ref{Grob:Bas}, $1\not\in \mathcal{I}(K_n,K_4)$.  Additionally $E(G)\subset E(K_n)$, and thus by
		 Lemma \ref{lem:SubLem}, $1\not\in \mathcal{I}(G,K_4)$.  Therefore $\chi_{alg}(G)\leq 4$.	
	\end{proof}

\begin{prob} We do not know the complexity of deciding if $\chi_{alg}(G) =4$, i.e., of deciding if $1 \in \cl I(G, K_3)$.
\end{prob}

\section{The Locally Commuting Algebra}
\label{sec:loc}

Our analysis of  ``algebraic colorability'' shows $\chi_{alg}$
 is too coarse a parameter to provide much information
about graphs, after all every graph will be algebraically colorable by at most 4 colors, and so has $\chi_{alg}$ has very little in common with any of the quantum chromatic numbers. 
This problem does not occur with the hereditary chromatic number, which may be equal to the usual quantum chromatic number for all graphs,
but unfortunately determining the elements of the hereditary closure of an ideal is difficult.
In the forthcoming  section, we add further physically motivated algebraic relations to  $\cl I(G,H)$, in order to obtain a chromatic number that is more amenable to algebraic analysis while still retaining a quantum flavor.

The new relations are additional commutation relations in $\cl A(G,H)$ and yield
a new algebra $\cl A_{lc}(G,H)$ that we call the {\it locally commuting algebra}.
This yields a new type of chromatic number, $\chi_{lc}$.
One goal of this section is to prove $\chi_{lc}(K_n)=n$.

Since our algebras were initially motivated by quantum chromatic numbers, it is natural to look to quantum mechanics for further relations to impose. In the case of a graph, we can imagine each vertex as corresponding to a laboratory and think of two vertices as connected whenever those laboratories can conduct a joint experiment. In this case, all of the measurement operators for the two labs should commute, i.e., whenever $(v,w)$ is an edge, then the commutator $[e_{v,i}, e_{w,j}] := e_{v,i}e_{w,j} - e_{w,j}e_{v,i} =0$. Note that this commutation rule is exactly the rule that we were able to derive in the case of three colors in Proposition~\ref{prop:G-K-3-commute}.
This motivates the following definitions.

\begin{defn}  
Let $\cl G= (I, O, \lambda)$ be a synchronous game with $|I|=n$ and $|O|=m$. We say that $v,w \in I$ are \df{adjacent} and write $v \sim w$ provided that $v \ne w$ and there exists $a,b \in O$ such that $\lambda(v,w,a,b) =0$.  We define the \df{locally commuting ideal} of the game to be the 2-sided ideal $\cl I_{lc}(\cl G)$ in $\bb C[\bb F(n,m)]$ generated by the set
\[ \{ e_{v,a}e_{w,b} \mid\lambda(v,w,a,b)=0 \} \cup \{ [e_{v,a}, e_{w,b}] \mid \ v \sim w, \, \forall a,b \in O \}. \]
We set $\cl A_{lc}(\cl G) = \bb C[ \bb F(n,m)]/\cl I_{lc}(\cl G)$ and call this the \df{locally commuting algebra of $\cl G$}.

In the case that $G$ and $H$ are graphs and $\cl G$ is the graph homomorphism game from $G$ to $H$ we set $$\cl I_{lc}(G,H)= \cl I_{lc}(\cl G)$$ and $$\cl A_{lc}(G,H) = \cl A_{lc}(\cl G).$$	 We write $G \stackrel{lc}{\to} H$ provided that $\cl I_{lc}(G,H) \ne \bb C[\bb F(n,m)]$ and set
\[ \chi_{lc}(G) = \min \{ c \mid G \stackrel{lc}{\to} K_c \}.\]
We similarly define
\[ \omega_{lc}(G) = \max \{ c \mid K_c \stackrel{lc}{\to} G \}.\]
\end{defn}

Note that in the case of the graph homomorphism game from $G$ to $H$ we have that $I=V(G)$ and $v \sim w \iff (v,w) \in E(G)$. Thus, the relationship $\sim$ extends the concept of adjacency to  the inputs of a general synchronous game.
 
Thus,  $\cl A_{lc}(G, K_c)$ is the universal *-algebra generated by self-adjoint projections $\{ E_{v,i}: v \in V(G), 1 \le i \le c \}$ satisfying
\begin{itemize}
\item $\sum_{i=1}^c E_{v,i} = I, \forall v,$
\item  $v \sim w \implies E_{v,i}E_{w,i} =0, \forall i,$
\item $v \sim w \implies [E_{v,i}, E_{w,j}]=0, \forall i,j$ \end{itemize}
and $\chi_{lc}(G)$ is the least $c$ for which such a non-trivial *-algebra exists.

We begin by showing that every graph homomorphism, in the usual sense, yields an $lc$-morphism:

\begin{lemma}
\label{l:to->lc}
If $G\to H$, then $G\stackrel{lc}{\to} H$.
\end{lemma}
\begin{proof}
Let $\phi\colon G\to H$ be a graph homomorphism. We must show $\cl A_{lc}(G,H)\neq0$. Consider the map $\cl A_{lc}(G,H)\to\bb C$ sending $e_{v,\phi(v)}$ to 1 and $e_{v,x}$ to 0 for $x\neq\phi(v)$. It is easy to see this is a well-defined $\bb C$-algebra map and hence surjective. As a result, $\cl A_{lc}(G,H)\neq0$.
\end{proof}

\begin{cor}
\label{cor:chi-lc>=chi}
We have $\chi_{lc}(G)\leq\chi(G)$ and $\omega(G)\leq\omega_{lc}(G)$.
\end{cor}
\begin{proof}
There is a graph homomorphism $G\to K_{\chi(G)}$ so by Lemma \ref{l:to->lc}, we have $G\stackrel{lc}{\to} K_{\chi(G)}$ and hence $\chi_{lc}(G)\leq\chi(G)$. The inequality for $\omega$ is shown in an analogous fashion.
\end{proof}

We are now ready to prove the main result of this section.

\begin{thm} \label{thm:chi_lc(K_n)}
For every $n \ge 1,$ we have that $\chi_{lc}(K_n) = n$.
\end{thm}
\begin{proof} The case $n=1$ is trivial to check.  Assume that $n \ge 2$. Let us label the vertices by numbers $1$ to $n$ and set $c= \chi_{lc}(K_n).$ Then $c \le n$ and we want to show that $c < n$ is impossible.

Suppose that $c <n$, then since $I = \sum_{i=1}^c E_{k,i}$ for $1 \le k \le n$ we have that
\[  I = \prod_{k=1}^n \big( \sum_{i=1}^c E_{k,i} \big) = \sum_{i_1,...,i_n} E_{1,i_1}E_{2,i_2} \cdots E_{n, i_n}, \]
where the sum is over all $n$-tuples with $i_j \in \{ 1,..., c\}$.  By the locally commuting hypothesis, all of the above projections commute, so we may re-order the sum in any fashion.
Since $c < n$, by the pigeon-hole principle, each $n$-tuple must contain $j,l$ with $i_j = i_l = h$. But then $E_{j,h}E_{l,h}=0$. Hence, each product occurring in the above sum is 0 and so the sum is 0. Thus, we have that $I =0$, which shows that $\cl A_{lc}(K_n, K_c) =(0)$, that is,  $\cl I_{lc}(K_n,K_c) = \bb C[\bb F(n,c) ]$ for $c<n$.
\end{proof}

\begin{prob} We do not know if the Lov\"asz sandwich result holds in this context, i.e., if $\omega_{lc}(G) \le \vartheta(\overline{G}) \le \chi_{lc}(G)$.
\end{prob}


\section{Some basic properties of $\cl A_{lc}$ and $\chi_{lc}$}
\label{sec:loc properties}

In this section we analyze the algebra $\cl A_{lc}(G,H)$ more closely and obtain the value of $\chi_{lc}(G)$ for a few select graphs. In particular, we are able to show, indirectly, that $\chi_{lc} \ne \chi_q$.  Throughout this section we shall write $\simeq$ to indicate that two algebras are isomorphic.  We shall use $\bb C^n$ to denote the abelian algebra of complex-valued functions on $n$ points.

It is easy to check that $\cl A_{lc}(G,H)$ is the quotient of $\mathbb{C}\<e_{vx}\mid v\in G, x\in H\>$ by the ideal generated by the following relations:
\begin{enumerate}
\item $\sum_{x\in H}e_{vx}=1$,
\item $e_{vx}^2=e_{vx}$,
\item $e_{vx}e_{vy}=0$ for $x\neq y$, 
\item $e_{vx}e_{wy}=0$ if $v\sim w$ and $x\not\sim y$, and
\item $[e_{vx},e_{wy}]=0$ for $v\sim w$.
\end{enumerate}
In Lemma \ref{l:to->lc} we showed that graph homomorphisms induce $lc$-morphisms. We next show that we can ``compose'' $lc$-morphisms.

\begin{lemma}
\label{l:lc-comp}
If $G\stackrel{lc}{\to} H$ and $H\stackrel{lc}{\to} K$, then $G\stackrel{lc}{\to} K$.
\end{lemma}
\begin{proof}
If $\cl A_{lc}(G,H)$ and $\cl A_{lc}(H,K)$ are non-zero, then we must prove that $\cl A_{lc}(G,K)$ is non-zero as well. To see this, consider the map
\[  \mathbb{C}\<e_{vr}\mid v\in G, r\in K\> \to  \cl A_{lc}(G,H)\otimes \cl A_{lc}(H,K) \]
given by
\[
e_{vr}\mapsto \sum_{x\in H} e_{vx}\otimes e_{xr}
\]
and suppose that it vanishes on $\cl I_{lc}(G,K)$. Hence, there would be a well-defined map on the quotient,
\[
\cl A_{lc}(G,K)\to \cl A_{lc}(G,H)\otimes \cl A_{lc}(H,K)
\]
\[
e_{vr}\mapsto \sum_{x\in H} e_{vx}\otimes e_{xr}.
\]
If $1=0$ in $\cl A_{lc}(G,K)$, then the same would be true in $\cl A_{lc}(G,H)\otimes \cl A_{lc}(H,K)$, since this map sends units to units.

Thus it remains to show that the above map vanishes on $\cl I_{lc}(G,K)$. In order to do this, it is sufficient to check that each generating relation is sent to zero.  This is easily checked, for example,


\[
\sum_{r\in K}\sum_{x\in H} e_{vx}\otimes e_{xr}=\sum_{x\in H} e_{vx}\otimes \sum_{r\in K}e_{xr}=\sum_{x\in H} e_{vx}\otimes 1 = 1.
\]
Checking the other relations is left to the reader.
\end{proof}

\begin{cor}
\label{cor:chi_lc-under-maps}
If $G\stackrel{lc}{\to} H$, then $\chi_{lc}(G)\leq\chi_{lc}(H)$.
\end{cor}
\begin{proof}
Let $c=\chi_{lc}(H)$. Then we have $H\stackrel{lc}{\to} K_c$ and hence $G\stackrel{lc}{\to} K_c$. Thus, $\chi_{lc}(G)\leq c=\chi_{lc}(H)$.
\end{proof}

We also have the following consequence of the proof of Lemma \ref{l:lc-comp}.
\begin{thm}
\label{thm:Alc-functor}
The assignment
\[
(\!\textrm{Graphs})\times(\!\textrm{Graphs})\longrightarrow (\mathbb{C}\textrm{-algebras})
\]
\[
(G,H)\longmapsto \cl A_{lc}(G,H)
\]
is a functor, which is covariant in the first factor and contravariant in the second; note that the category of graphs is with usual morphisms, not $lc$-morphisms.
\end{thm}
\begin{proof}
If $\phi:G\to G'$ is a morphism, then we have a map $\cl A_{lc}(G,H)\to\cl A_{lc}(G',H)$ given by $e_{v,x}\mapsto e_{\phi(v),x}$. On the other hand, if $\phi:H\to K$ is a morphism, then we have $H\stackrel{lc}{\to}K$ and so from the proof of Lemma \ref{l:lc-comp}, we have
\[
\cl A_{lc}(G,K)\to\cl A_{lc}(G,H)\otimes\cl A_{lc}(H,K).
\]
Since $\phi$ is a morphism of graphs, we have a map $\cl A_{lc}(H,K)\to\mathbb{C}$ as in the proof of Lemma \ref{l:to->lc}. Composing with the above, we have $\cl A_{lc}(G,K)\to\cl A_{lc}(G,H)$. Explicitly, this map is given by sending $e_{vx}\in\cl A_{lc}(G,K)$ to $\sum_{\phi(r)=x}e_{vr}$.
\end{proof}

We now show how the functor $\cl A_{lc}$ interacts with various natural graph operations. To begin, recall that if $G$ is a graph, its suspension $\Sigma G$ is defined by adding a new vertex $v$ and an edge from $v$ to each of the vertices of $G$.

Given an algebra $\cl A$ we shall let $\cl A^c$ denote the algebra of $c$-tuples with entries from $\cl A$, i.e., the tensor product $\cl A \otimes \bb C^c\simeq\cl A^{\oplus c}$ where $\bb C^c$ can be identfied with the algebra of $\bb C$-valued functions on $c$ points.

\begin{prop}
\label{prop:sigmaG-H-lc}
Let $G$ and $H$ be any graphs, and let $H_{ni}$ be the non-isolated vertices. For $y\in H_{ni}$ we let $N_y$ denote the neighborhood of $y$, i.e., the induced subgraph of $H$ with vertices adjacent to $y$; notice $y\notin N_y$ unless $y$ has a self-edge. Then
\[
\cl A_{lc}(\Sigma G,H)\simeq\bigoplus_{y\in H_{ni}}\cl A_{lc}(G,N_y).
\]
In particular, if $H$ is vertex transitive and $y$ is any vertex of $H$ with neighborhood $N$, then
\[
\cl A_{lc}(\Sigma G,H)\simeq\cl A_{lc}(G,N)^{|H|}.
\]
\end{prop}
\begin{proof}
Let $u$ be the new vertex added to $\Sigma G$, i.e., $u\in \Sigma G\setminus G$. Since $u$ is adjacent to every vertex of $G$, we see $e_{ux}$ commutes with $e_{vy}$ for all $v\in G$ and $x,y\in H$. Furthermore, the defining relations of $\cl A_{lc}$ tell us $e_{ux}e_{uy}=\delta_{x,y}e_{ux}$ where $\delta$ denotes the Kronecker delta function. So,
\[
\cl A_{lc}(\Sigma G,H) \simeq \cl A_{lc}(G,H)[e_{ux}]/(\sum_x e_{ux}-1,e_{ux}e_{uy}=\delta_{x,y}e_{ux}).
\]
In other words, the $e_{ux}$ for $x\in H$ are commuting orthogonal idempotents, which shows
\[
\cl A_{lc}(\Sigma G,H) \simeq \bigoplus_{y\in H} \cl A_{lc}(G,H)e_{uy} \simeq \bigoplus_{y\in H} \cl A_{lc}(G,H)/(e_{vx}\colon x\not\sim y),
\]
where the last equality comes from the fact that $e_{vx}e_{uy}=0$ for $x\not\sim y$.

Now note that $e_{vx}$ remains non-zero in the quotient $\cl A_{lc}(G,H)/(e_{vx}\colon x\not\sim y)$ if and only if $x\sim y$. Thus,
\[
\cl A_{lc}(G,H)/(e_{vx}\colon x\not\sim y)\simeq\cl A_{lc}(G,N_y),
\]
which establishes the first assertion of the proposition. The second assertion easily follows from the first since all neighborhoods are isomorphic.
\end{proof}

\begin{cor}
\label{cor:sigmaG-lc}
For all graphs $G$, we have $\cl A_{lc}(\Sigma G,K_1)=0$. If $c\ge 2$, then
\[
\cl A_{lc}(\Sigma G,K_c)\simeq\cl A_{lc}(G,K_{c-1})^c.
\]
\end{cor}
\begin{proof}
This is an immediate consequence of Proposition \ref{prop:sigmaG-H-lc} using that $K_c$ is vertex transitive.
\end{proof}

\begin{cor}\label{cor:sigmaG-lc-chrom} For all graphs $G$, we have $\chi_{lc}(\Sigma G) = \chi_{lc}(G) +1$.
\end{cor}
\begin{proof} By the above isomorphism, the least $c+1$ such that $\cl A_{lc}(\Sigma G, K_{c+1}) \ne (0)$ is equal to the least $c$ such that $\cl A_{lc}(G, K_c) \ne (0)$.
\end{proof}

\begin{remark}  In \cite{MR3} an example of a graph $G$ is given for which $\chi_q( \Sigma G) = \chi_q(G)$. Hence, either $\chi_{lc}(\Sigma G) \ne \chi_q(\Sigma G)$ or $\chi_{lc}(G) \ne \chi_q(G)$. 
\end{remark}

\begin{cor}
\label{cor:lc-Kc}
If $c\ge n$, then
\[
\cl A_{lc}(K_n,K_c)\simeq \mathbb{C}^{c(c-1)\dots(c-n+1)}.
\]
If $c<n$, then $\cl A_{lc}(K_n,K_c)=0$.
\end{cor}
\begin{proof}
One easily checks that $\cl A_{lc}(K_1,G)\simeq{\bb C}^{|G|}$ for any graph $G$. In particular, our desired statement holds for $n=1$. The proof then follows from induction on $n$ by applying Corollary \ref{cor:sigmaG-lc} and using that $K_n=\Sigma K_{n-1}$.
\end{proof}

\begin{remark}
In Theorem \ref{thm:chi_lc(K_n)}, we proved $\chi_{lc}(K_n)=n$. Corollary \ref{cor:lc-Kc} gives another proof of this result which is more refined: the corollary tells us the specific structure of $\cl A_{lc}(K_n,K_c)$ whereas the theorem merely tells us it is non-zero.
\end{remark}

Using Proposition \ref{prop:sigmaG-H-lc}, we can easily understand iterated suspensions. In particular, we can understand \emph{any} $lc$-map out of $K_n$:
\begin{cor}
\label{cor:lc-out-of-K_n}
If $H$ is a graph, then $\cl A_{lc}(K_1,H)=\CC^{|H|}$ and for $n>1$, we have
\[
\cl A_{lc}(K_n,H) \simeq \bigoplus_{S\subseteq H}\CC^{ \left| N_S\right| (n-1)!}
\]
where $S$ is an $(n-1)$-clique and $N_S=\{z\in H\mid z\sim x\ \forall x\in S\}$. 
In particular, if $H$ is vertex transitive on $c$ vertices, then
\[
\cl A_{lc}(K_n,H) \simeq \CC^{c(c-1)\dots(c-n+2)|N_S|}
\]
where $S$ is any $(n-1)$-clique in $H$.
\end{cor}
\begin{proof}
We leave the $n=1$ case to the reader. Iteratively applying Proposition \ref{prop:sigmaG-H-lc}, we see
$$
\cl A_{lc}(K_n,H)
\simeq \bigoplus_{x\in H}\cl A_{lc}(K_{n-1},N_x) \simeq \dots
\simeq \bigoplus_{(x_{n-1},\dots,x_2,x_1)}\!\!\!\!\cl A_{lc}(K_1,N_{x_{n-1}}\dots N_{x_2}N_{x_1})
$$
where the index of the direct sum runs over all sequences $(x_{n-1},\dots,x_2,x_1)$ with $x_{i+1}\in N_{x_i}N_{x_{i-1}}\dots N_{x_1}$.

We show by induction that the $x_1,\dots,x_i$ form an $i$-clique and that $N_{x_i}N_{x_{i-1}}\dots N_{x_1}=N_{\{x_1,\dots,x_i\}}$. For $i=1$ this is just the definition. For $i>1$, observe that by construction $x_i\in N_{x_{i-1}}\dots N_{x_1}=N_{\{x_1,\dots,x_{i-1}\}}$ and since $x_1,\dots,x_{i-1}$ forms an $(i-1)$-clique, we see $x_1,\dots,x_i$ forms an $i$-clique. Next, $N_{x_i}N_{\{x_1,\dots,x_{i-1}\}}$ is the set of $z\in N_{\{x_1,\dots,x_{i-1}\}}$ that are adjacent to $x_i$, which is the definition of $N_{\{x_1,\dots,x_i\}}$.

This shows
\[
\cl A_{lc}(K_n,H) \simeq \bigoplus_{(x_{n-1},\dots,x_2,x_1)} A_{lc}(K_1,N_{\{x_1,\dots,x_{n-1}\}}),
\]
and by the $n=1$ case, each summand is isomorphic to $|N_{\{x_1,\dots,x_{n-1}\}}|$ copies of $\CC$. Now notice that $N_{\{x_1,\dots,x_{n-1}\}}$ is independent of the order of the sequence, and hence this term arises $(n-1)!$ times. This yields the desired statement.

If $H$ is vertex transitive and contains an $(n-1)$-clique, then there are ${c\choose n-1}$ such choices of an $(n-1)$-clique. 
So, we end up with $${c\choose n-1}(n-1)!|N_S|=c(c-1)\dots(c-n+2)|N_S|$$ 
many copies of $\CC$.
\end{proof}

We next consider how $\cl A_{lc}$ interacts with the categorical product $\times$. Recall that if $H$ and $K$ are graphs then $H\times K$ is a graph with vertex set $V(H)\times V(K)$ and where $(v,x)\sim (w,y)$ if and only if $v\sim_H w$ and $x\sim_K y$.

\begin{thm}
\label{thm:preserve-prod}
We have a natural isomorphism
\[
\cl A_{lc}(G,H\times K)\simeq \cl A_{lc}(G,H)\otimes\cl A_{lc}(G,K),
\]
where $\times$ is the categorical product in graphs.
\end{thm}
\begin{remark}
\label{rmk:preserve-prod}
As shown in Theorem \ref{thm:Alc-functor}, we can view $\cl A_{lc}$ as a functor. Theorem \ref{thm:preserve-prod} can be interpreted as saying that the second factor of the functor $\cl A_{lc}(-,-)$ preserves products. Recall that since the second factor is contravariant, preserving products means that it takes products in graphs to coproducts in $\mathbb{C}$-algebras, namely tensor products.
\end{remark}
\begin{proof}
Since we have a map $H\times K\to H$, Theorem \ref{thm:Alc-functor} tells us that we have a map $\cl A_{lc}(G,H)\to\cl A_{lc}(G,H\times K)$. Similarly for $K$, and hence a natural map $\cl A_{lc}(G,H)\otimes\cl A_{lc}(G,K)\to\cl A_{lc}(G,H\times K)$. Explicitly, it is given by
\[
e_{\alpha,v}\otimes 1\mapsto\sum_{x\in K} e_{\alpha,(v,x)},\quad 1\otimes e_{\alpha,y}\mapsto\sum_{w\in H} e_{\alpha,(w,y)}.
\]
We now construct an inverse map given by
\[
\cl A_{lc}(G,H\times K)\to \cl A_{lc}(G,H)\otimes\cl A_{lc}(G,K)
\]
\[
e_{\alpha,(v,x)}\mapsto e_{\alpha,v}\otimes e_{\alpha,x}.
\]
Provided this is well-defined, it is indeed an inverse since
\[
e_{\alpha,v}\otimes e_{\alpha,x}\mapsto \sum_{y,w}e_{\alpha,(v,y)}e_{\alpha,(w,x)}=e_{\alpha,(v,x)},
\]
where the last equality uses the fact that $e_{\alpha,(v,y)}e_{\alpha,(w,x)}=0$ unless $(v,y)=(w,x)$. Also, composing the two maps in the opposite order yields the identity since
\[
e_{\alpha,v}\otimes1\mapsto \sum_x e_{\alpha,(v,x)}\mapsto \sum_xe_{\alpha,v}\otimes e_{\alpha,x}=e_{\alpha,v}\otimes\sum_x e_{\alpha,x}=e_{\alpha,v}\otimes1.
\]

The rest of the proof is devoted to showing that the map we constructed above is well-defined. First note that
\[
(e_{\alpha,v}\otimes e_{\alpha,x})^2=e_{\alpha,v}^2\otimes e_{\alpha,x}^2=e_{\alpha,v}\otimes e_{\alpha,x}.
\]
Next,
\[
\sum_{(v,x)\in H\times K}e_{\alpha,v}\otimes e_{\alpha,x}=\sum_v e_{\alpha,v}\otimes \sum_x e_{\alpha,x}=1\otimes1.
\]
If $(v,x)\neq(w,y)$ then without loss of generality $v\neq w$. So,
\[
e_{\alpha,(v,x)}e_{\alpha,(w,y)}\mapsto (e_{\alpha,v}\otimes e_{\alpha,x})(e_{\alpha,w}\otimes e_{\alpha,y})=0\otimes e_{\alpha,x}e_{\alpha,y}=0.
\]
Next assume $\alpha\sim\beta$. We need to check that the image of $e_{\alpha,(v,x)}e_{\beta,(w,y)}$ is equal to that $e_{\beta,(w,y)}e_{\alpha,(v,x)}$.
\[
e_{\alpha,(v,x)}e_{\beta,(w,y)}\mapsto (e_{\alpha,v}\otimes e_{\alpha,x})(e_{\beta,w}\otimes e_{\beta,y})=e_{\alpha,v}e_{\beta,w}\otimes e_{\alpha,x}e_{\beta,y}
\]
and since $\alpha\sim\beta$, this is equal to
\[
e_{\beta,w}e_{\alpha,v}\otimes e_{\beta,y}e_{\alpha,x}=(e_{\beta,w}\otimes e_{\beta,y})(e_{\alpha,v}\otimes e_{\alpha,x})
\]
which is the image of $e_{\beta,(w,y)}e_{\alpha,(v,x)}$. Finally, we must show that if $\alpha\sim\beta$ and $(v,x)\not\sim(w,y)$ then $e_{\alpha,(v,x)}e_{\beta,(w,y)}$ maps to 0. Since $(v,x)\not\sim(w,y)$, without loss of generality $v\not\sim w$. Then the image of $e_{\alpha,(v,x)}e_{\beta,(w,y)}$ is
\[
e_{\alpha,v}e_{\beta,w}\otimes e_{\alpha,x}e_{\beta,y}=0
\]
since $e_{\alpha,v}e_{\beta,w}=0$.
\end{proof}

\begin{prob}
In light of Remark \ref{rmk:preserve-prod}, we ask if the functor $\cl A_{lc}(G,-)$ preserves all finite limits. Given Theorem \ref{thm:preserve-prod}, this is equivalent to asking if it preserves equalizers. 
\end{prob}
\begin{prob}[{Yoneda-type question}]
Does the functor $\cl A_{lc}(G,-)$ determine $G$?
\end{prob}

Recall that the exponential of graphs $K^H$ is defined as follows: its vertex set consists of all functions $f\colon V(H)\to V(K)$ and there is an edge $f\sim_{K^H} g$ if for all $v\sim_H w$ we have $f(v)\sim_K g(w)$. Note that if $f$ is a graph homomorphism, then it has a self-edge.

Within the category of graphs, the product is left adjoint to exponentiation, that is
\[
\Hom(G\times H, K) = \Hom(G,K^H).
\]
One can ask if this adjunction remains true for graphs with $lc$-morphisms.


\begin{prob}
\label{prob:adjunction-exp}
If one allows $H$ to have self edges, then $\cl A_{lc}(G,H)$ can be defined using the same relations at the beginning of this section. One can then define a map
\[
\cl A_{lc}(G\times H,K)\to \cl A_{lc}(G,K^H)
\]
given by
\[
e_{(\alpha,v),x}\longmapsto \sum_{f(v)=x}e_{\alpha,f}.
\]
Is this map an isomorphism? In other words, are products and exponentials adjoints?
\end{prob}

In addition to the categorical product $\times$, there are several other kinds of products on graphs, which we now consider. Both $G\cart H$ and $G\boxtimes H$ have vertex set $V(G)\times V(H)$. In the former, $(v,x)\sim (w,y)$ if and only if $v=w$ and $x\sim y$, or $v\sim w$ and $x=y$. In the latter, $(v,x)\sim (w,y)$ if and only if $v\sim w$ and $x\sim y$, or $v=w$ and $x\sim y$, or $v\sim w$ and $x=y$. The products $\cart$ and $\boxtimes$ are referred to as the Cartesian and strong products, respectively.

For any pair of graphs we have $\chi(G\cart H)=\max \{\chi(G), \chi(H) \}$. We show that the same is true for $\chi_{lc}$.

\begin{thm}[{$\chi_{lc}$ of Cartesian product}]
\label{thm:Cart-prod-lc}
For any graphs $G$ and $H$, we have
\[
\chi_{lc}(G\cart H)=\max\{\chi_{lc}(G),\chi_{lc}(H)\}.
\]
\end{thm}
\begin{proof}
We have at least $|H|$ maps $G\to G\cart H$, so Lemma \ref{l:to->lc} and Corollary \ref{cor:chi_lc-under-maps} show $\chi_{lc}(G)\leq\chi_{lc}(G\cart H)$. Similarly for $H$ and so $\max\{\chi_{lc}(G),\chi_{lc}(H)\}\leq\chi_{lc}(G\cart H)$. To prove the result, it now suffices to show we have a map
\[
\cl A_{lc}(G\cart H,K_c)\to \cl A_{lc}(G,K_c)\otimes \cl A_{lc}(H,K_c).
\]
Indeed, if $\cl A_{lc}(G,K_c)$ and $\cl A_{lc}(H,K_c)$ are non-zero, then so is $\cl A_{lc}(G\cart H,K_c)$ since the above map would send 0 to 0 and 1 to 1, and if $0=1$ in $\cl A_{lc}(G\cart H,K_c)$, then $0=1$ in $\cl A_{lc}(G,K_c)\otimes \cl A_{lc}(H,K_c)$, which is not the case. Taking $c=\max\{\chi_{lc}(G),\chi_{lc}(H)\}$, this would then show $\chi_{lc}(G\cart H,K_c)\geq c$.

We now construct the above map. We define it by:
\[
e_{(x,y),k}\mapsto \sum_{i\in\ZZ/c} e_{x,i}\otimes e_{y,k-i}
\]
and show it is well-defined. First suppose that $(x,y)\sim(x',z)$ and $k\not\sim\ell$. Then $k=\ell$ and without loss of generality $x=x'$ and $y\sim z$. Then
\[
e_{(x,y),k}e_{(x,z),\ell}\mapsto \sum_{i,j}e_{x,i}e_{x,j}\otimes e_{y,k-i}e_{z,\ell-j}=0
\]
since $e_{x,i}e_{x,j}=0$ if $i\neq j$, and if $i=j$, then $k-i=\ell-j$ and so $e_{y,k-i}e_{z,\ell-j}=0$.

Next, if $y\sim z$, then the images of $e_{(x,y),k}e_{(x,z),\ell}$ and $e_{(x,z),\ell}e_{(x,y),k}$ are equal since
\[
e_{(x,y),k}e_{(x,z),\ell}\mapsto \sum_{i,j}e_{x,i}e_{x,j}\otimes e_{y,k-i}e_{z,\ell-j}
\]
and $e_{y,k-i}e_{z,\ell-j}=e_{z,\ell-j}e_{y,k-i}$ as $y\sim z$, and $e_{x,i}e_{x,j}=\delta_{ij}e_{x,i}=e_{x,j}e_{x,i}$.

We next see that
\[
\sum_k e_{(x,y),k}\mapsto \sum_i \sum_k e_{x,i}\otimes e_{y,k-i}=\sum_i \sum_k e_{x,i}\otimes e_{y,k}=\sum_i e_{x,i}\otimes \sum_k e_{y,k}=1\otimes1.
\]
If $k\neq\ell$, then
\[
e_{(x,y),k}e_{(x,y),\ell}\mapsto \sum_{i,j}e_{x,i}e_{x,j}\otimes e_{y,k-i}e_{y,\ell-j}=0
\]
since $e_{x,i}e_{x,j}=0$ if $i\neq j$, and if $i=j$, then $k-i\neq \ell-j$ and so $e_{y,k-i}e_{y,\ell-j}=0$.

Lastly,
\[
e_{(x,y),k}^2\mapsto \sum_{i,j}e_{x,i}e_{x,j}\otimes e_{y,k-i}e_{y,k-j}=\sum_{i}e_{x,i}^2\otimes e_{y,k-i}^2
\]
since $e_{x,i}e_{x,j}=0$ if $i\neq j$. Thus, $e_{(x,y),k}^2$ and $e_{(x,y),k}$ have the same image. This completes the proof that the map is well-defined.
\end{proof}

\begin{lemma}
\label{l:prod-map}
Given $G\lc K$ and $H\lc K'$, we have $G\cdot H\lc K\cdot K'$ for any $\cdot\in\{\times,\cart,\boxtimes\}$.
\end{lemma}
\begin{proof}
It suffices to construct a map
\[
\cl A_{lc}(G\cdot H,K\cdot K')\to\cl A_{lc}(G,K)\otimes\cl A_{lc}(H,K').
\]
We define it by
\[
e_{(x,y),(k,k')}\mapsto e_{x,k}\otimes e_{y,k'}.
\]
One readily checks that this map is well-defined. For example, in the case of the Cartesian product $\square$ we show that if $(x,y)\sim(z,w)$ and $(k,k')\not\sim(\ell,\ell')$, then $e_{(x,y),(k,k')}e_{(z,w),(\ell,\ell')}$ maps to 0. Without loss of generality, we can assume that $x=z$ and $y\sim w$. Then the image is $e_{x,k}e_{x,\ell}\otimes e_{y,k'}e_{w,\ell'}$, which is automatically 0 if $k\neq\ell$. So, we may assume $k=\ell$, in which case $k'\not\sim\ell'$ since $(k,k')\not\sim(\ell,\ell')$. But then $e_{y,k'}e_{w,\ell'}=0$.
\end{proof}

\begin{cor}\label{cor:box-prod-lc-bound}
We have $\chi_{lc}(G\boxtimes H)\leq\chi_{lc}(G)\chi_{lc}(H)$.
\end{cor}
\begin{proof}
This follows immediately from Lemma \ref{l:prod-map} after observing that $K_n\boxtimes K_m=K_{nm}$.
\end{proof}

In \cite{CMRSSW}, the authors showed that  $8=\chi(C_5 \boxtimes K_3) = \chi_q(C_5 \boxtimes K_3) > \chi_{vect}(C_5 \boxtimes K_3)=7$. Thus, separating $\chi_{q}$ from $\chi_{vect}$. Later, \cite{PSSTW} showed that $\chi_{qc}(C_5 \boxtimes K_3) =8$, separating the potentially smaller $\chi_{qc}$ from $\chi_{vect}$.  We show below that $\chi_{lc}(C_5 \boxtimes K_3)= 8$ as well.

Before considering $C_5\boxtimes K_3$ we begin with a simpler example.

\begin{exam}[{$C_5\boxtimes K_2$}]
\label{ex:C_5-boxtimes-K_2}
Let $G=C_5\boxtimes K_2$. It is easy to see that $\omega(G)=4$ and $\chi(G)=5$, so a priori $\chi_{lc}$ could be 4 or 5. We show
\[
\chi_{lc}(C_5\boxtimes K_2)=5.
\]
We need to show that $\cl A_{lc}(G,K_4)=0$. The graph $G$ is made up of 2 pentagons stacked on top of each other. Let one of the pentagons have vertices $x,y,z,w,s$ labeled clockwise and let the other pentagon have vertices $x',y',z',w',s'$ with $x$ and $x'$ having the same neighbors. For ease of notation, we denote $e_{v,i}$ by $v_i$. Note that
\[
1=\sum_{\sigma\in S_4}(s_{\sigma(3)}s'_{\sigma(4)}+s'_{\sigma(3)}s_{\sigma(4)})x_{\sigma(1)}x'_{\sigma(2)}y_{\sigma(3)}y'_{\sigma(4)}(z_{\sigma(1)}z'_{\sigma(2)}+z'_{\sigma(1)}z_{\sigma(2)}).
\]
Multiplying on both the left and right by $w_1w'_2$, we obtain
\[
w_1w'_2=w_1w'_2(s_3s'_4+s'_3s_4)(x_1x'_2+x'_1x_2)(y_3y'_4+y'_3y_4)(z_1z'_2+z'_1z_2)w_1w'_2=0.
\]
Similarly, we find $w_iw'_j=0$ for all $i,j$. As a result,
\[
1=\sum_{i,j}w_iw'_j=0
\]
and so $\cl A_{lc}(G,K_4)=0$.
\end{exam}

\begin{exam}[{$C_5\boxtimes K_3$}]
\label{ex:C_5-boxtimes-K_3}
Let $G=C_5\boxtimes K_3$. We see $\omega=6$ and $\chi=8$, so a priori $\chi_{lc}$ could be 6, 7, or 8. We show
\[
\chi_{lc}(C_5\boxtimes K_3)=8=\chi(C_5\boxtimes K_3).
\]
We must show $\cl A_{lc}(G,K_7)=0$. We follow the same notational conventions as in Example \ref{ex:C_5-boxtimes-K_2}. Let $x,y,z,w,s$ be the vertices of $C_5$ labeled clockwise and denote the next two copies of $C_5$ by $x',\dots,s'$ resp.~$x'',\dots,s''$ where $x$, $x'$, $x''$ have the same neighbors in $G$. We also let $v_i=e_{v,i}$.

As in the previous example,
\[
w_1w'_2w''_3=w_1w'_2w''_3SXYZw_1w'_2w''_3,
\]
where $S=\sum_{i,j,k}s_is'_js''_k$ and analogously for $X,Y,Z$. We show that every term occurring in the sum on the righthand side of the above equation is 0. The indices $i,j,k$ occurring in the sum $S$ must all lie in $\{4,5,6,7\}$ otherwise the term vanishes (since it is multiplied by $w_1w'_2w''_3$). Our goal is to show that all terms in the sum in the righthand side vanish, so we can fix a summand in $S$ and assume $i,j,k$ equal $4,5,6$ respectively. Then the indices in $X$ must be 3 of $\{1,2,3,7\}$. We also see that the indices in $Z$ must be 3 of $\{4,5,6,7\}$. Fix a summands $x_ax'_bx''_c$ and $z_pz'_qz''_r$ of $X$ and $Z$, respectively. Then $\{1,2,\dots,7\}\setminus\{a,b,c,p,q,r\}$ has size at most 2. Therefore, every summand $t$ of $Y$ satisfies $XtZ=0$. So, $w_1w'_2w''_3=0$, and analogously we see $w_iw'_jw''_k=0$ for all $i,j,k$. So,
\[
1=\sum_{i,j}w_iw'_jw''_k=0
\]
showing that $\cl A_{lc}(G,K_7)=0$. As a result, $\chi_{lc}(G)=8$.
\end{exam}

We end by posing the following:
\begin{prob}
Since the definition of $\chi_{lc}$ is not obviously related to representations on Hilbert spaces, it is unclear how to relate it to $\chi_t$ for $t\in\{ loc, q, qa, qc, vect \}$. Where does $\chi_{lc}$ fit within this hierarchy?
\end{prob}

Even more specifically,
\begin{prob}
Give an example where $\chi_{lc}(G)\neq\chi(G)$.
\end{prob}

\end{document}